\documentclass[10pt]{amsart}

\usepackage{graphicx}
\usepackage{amsmath}
\usepackage{amssymb}

\def\RR{{\mathbb R}}
\def\MM{{\mathbb M}}

\def\NN{{\mathbb N}}
\def\OO{{\mathbb O}}

\def\SO{{\mathbb S\mathbb O}}
\def\SS{{\mathbb{S}}}
\def\Z{\mathcal{Z}}
\def\R{\mathcal{R}}
\def\Q{\mathcal{Q}}

\def\det{{\rm det}}
\def\Det{{\rm det}}

\def\Aff{{\rm Aff}}

\def\wto{\rightharpoonup}
\def\dto{{\scriptstyle\buildrel{\scriptstyle\longrightarrow}
\over{{}_{\scriptstyle\longrightarrow}}}}
\def\eps{\varepsilon}
\def\then{\Rightarrow}
\def\inff{\mathop{\inf\limits_{a\in\RR^N}}}
\def\infff{\mathop{\inff\limits_{b\in\RR^m}}}
\def\cupp{\mathop{\cup}}

\def\diag{{\rm diag}}
\def\rank{{\rm rank}}
\def\ell{l}

\newtheorem{theorem}{\bf Theorem}[section]
\newtheorem{lemma}[theorem]{\bf Lemma}

\newtheorem{corollary}[theorem]{\bf Corollary}

\newtheorem{problem}{\bf Problem}[section]
\newtheorem{question}{\bf Question}[section]

\theoremstyle{definition}
\newtheorem{definition}[theorem]{\bf Definition}

\theoremstyle{remark}
\newtheorem{remark}[theorem]{\bf Remark}

\title{Relaxation and 3d-2d passage theorems in hyperelasticity}

\author{\sc Omar Anza Hafsa}
\address{UNIVERSITE MONTPELLIER II, UMR-CNRS 5508, LMGC,  Place Eug\`ene Bataillon, 34095 Montpellier, France.}
\email{Omar.Anza-Hafsa@univ-montp2.fr}
\author{\sc Jean-Philippe Mandallena}
\address{Laboratoire MIPA (Math\'ematiques, Informatique, Physique et Applications)\newline UNIVERSITE DE NIMES, Site des Carmes, Place Gabriel P\'eri, 30021 N\^\i mes, France.}
\email{jean-philippe.mandallena@unimes.fr}


\begin{document}

\maketitle

\begin{abstract}
We give an overview of relaxation and 3d-2d passage theorems in hyperelasticity in the framework of the multidimensional calculus of variations. Some open questions are addressed. This paper, which is an expanded version of the outline-paper \cite{oah-jpm09b}, comes as a companion to \cite{oah-jpm09a}.
\end{abstract}

\tableofcontents


\section{Introduction}

These notes are concerned with the problems of relaxation and 3d-2d passage under determinant type constraints naturally related to hyperelasticity in the framework of the multidimensional calculus of variations. Our goal is, on the one hand,  to give an overview of our works (see \cite{oah-jpm06,oah-jpm07,oah-jpm08a,oah-jpm08b}) concerning these two problems, and, on the other hand, to highlight the fact that the  Dacorogna relaxation theorem (proved in 1982, see Theorem \ref{DacorognaTheorem2}) and the Le Dret-Raoul 3d-2d passage theorem (proved in 1993, see Theorem \ref{LeDret-Raoult-Theorem}) can be extented to theorems (see Theorems \ref{IRTnon-finite-case} and \ref{RD-s}) which are consistent (almost consistent for the relaxation problem) with the setting of hyperelasticity, whose the two basic conditions are:
\begin{itemize}
\item[(i)] the noninterpenetration of the matter and
\item[(ii)] the necessity of an infinite amount of energy to compress a finite volume of matter into zero volume.
\end{itemize}
 Despite the restriction on the polynomial growth of the energy density which is not compatible with (i) and (ii), the Dacorogna theorem provides the model of nonlinear relaxation theorems related to hyperelasticity. In Section 2, we show that this model theorem can be improved by introducing the class of {\em ample energy densities}, i.e., ``energy densities having a quasiconvexification which is of polynomial growth", see Definition \ref{Ample-Integrands} and Theorems \ref{IRTnon-finite-case} and \ref{TheoremA-B}, and we make clear the fact that the ample energy densities are consistent with (ii) (see \S 2.6). Similarly, in spite of the polynomial growth hypothesis on the energy density, the Le Dret-Raoult theorem provides the model of nonlinear dimension reduction theorems in hyperelasticity. In Section 3, we show that this theorem can be extended to the setting of ample energy densities (see Theorem \ref{RD-w}) as well as to the setting of hyperelasticity, i.e., to the case of energy densities which are compatible with (i) and (ii) (see Theorem \ref{RD-s}). This latter theorem gives an answer to the 3d-2d passage problem in hyperelasticity in the same spirit as the works of Ball (see \cite{Ball77}), Acerbi-Buttazzo-Percivale (see \cite{ABP91}) and  Friesecke-James-M\"uller (see \cite{FJM02}). It is the result of several works: mainly, the attempt of Percivale in 1991 (see \cite{percivale91}), the papers of Le Dret and Raoult (see \cite{ledret-raoult93,ledret-raoult95}) and the thesis of Ben Belgacem (see \cite{benbelgacem96}, see also \cite{benbelgacem97,benbelgacem00}).


\section{Relaxation theorems with determinant type constraints}

\subsection{Statement of the problem} Let $m,N\in\NN$ (with $\min\{m,N\}>1$), let $p>1$ and let $W:\MM^{m\times N}\to[0,+\infty]$ be Borel measurable and $p$-coercive, i.e., 
$$
\exists C>0\ \forall F\in\MM^{m\times N}\ W(F)\geq C|F|^p,
$$
 where $\MM^{m\times N}$ denotes the space of real $m\times N$ matrices. Define the functional $I:W^{1,p}(\Omega;\RR^m)\to[0,+\infty]$ by
$$
\displaystyle I(\phi):=\int_\Omega W(\nabla\phi(x))dx,
$$
where $\Omega\subset\RR^N$ is a bounded open set, and consider  $\overline{I}:W^{1,p}(\Omega;\RR^m)\to[0,+\infty]$ (the relaxed functional of $I$) given by 
$$
\displaystyle\overline{I}(\phi):=\inf\left\{\liminf_{n\to+\infty}I(\phi_n):\phi_n\stackrel{L^p}{\to}\phi\right\}.
$$
Denote the quasiconvex envelope of $W$ by $\Q W:\MM^{m\times N}\to[0,+\infty]$. The problem of the relaxation is the following.
\begin{problem}\label{problem1}
Prove (or disprove) that 
$$
\forall\phi\in W^{1,p}(\Omega;\RR^m)\ \ \overline{I}(\phi)=\int_\Omega \Q W(\nabla\phi(x))dx
$$
and find a representation formula for $\Q W$.
\end{problem}
At the begining of the eighties, in \cite{dacorogna82} Dacorogna answered to Problem \ref{problem1} in the case where $W$ is ``finite and without singularities" (see \S 2.2). Recently, in \cite{oah-jpm07,oah-jpm08b}  we extended the Dacorogna theorem as Theorem \ref{RTnon-finite-case} and Theorem \ref{IRTnon-finite-case} (see \S 2.3 and \S 2.4) and we showed that these theorems can be used to deal with Problem \ref{problem1} under the ``weak-Determinant Constraint", i.e., when $m=N$ and $W:\MM^{N\times N}\to[0,+\infty]$ is compatible with the following two conditions:
\begin{equation}\label{w-DC}
\left\{\begin{array}{l}W(F)=+\infty\iff -\delta\leq \det F\leq 0\hbox{ with $\delta\geq 0$ (possibly very large)}\\
W(F)\to+\infty\hbox{ as }\det F\to 0^+\end{array}\right.
\end{equation}
(see \S 2.6). However, an answer to Problem \ref{problem1} under the ``strong-Determinant Constraint", i.e., when $m=N$ and $W:\MM^{N\times N}\to[0,+\infty]$ is compatible with the two basic conditions of hyperelasticity:
\begin{equation}\label{s-DC}
\left\{\begin{array}{ll}W(F)=+\infty\iff \det F\leq 0&\hbox{ (non-interpenetration of matter)}\\
W(F)\to+\infty\hbox{ as }\det F\to 0^+&\left(\begin{array}{l}\hbox{necessity of an infinite amount}\\
\hbox{of energy to compress a finite}\\
\hbox{volume into zero volume}\end{array}\right)\end{array}\right.,
\end{equation}
is still unknown (see \S 2.7).


\subsection{Representation of $\Q W$  and $\overline{I}$: finite case} Let $\Z_\infty W, \Z W:\MM^{m\times N}\to[0,+\infty]$ be respectively defined by:
\begin{itemize}
\item[\SMALL$\blacklozenge$] $\displaystyle \Z_\infty W(F):=\inf\left\{\int_Y W(F+\nabla\varphi(y))dy:\varphi\in W^{1,\infty}_0(Y;\RR^m)\right\}$;
\item[\SMALL$\blacklozenge$] $\displaystyle\Z W(F):=\inf\left\{\int_Y W(F+\nabla\varphi(y))dy:\varphi\in\Aff_0(Y;\RR^m)\right\}$,
\end{itemize}
where $Y:=]0,1[^N$, $W^{1,\infty}_0(Y;\RR^m):=\{\varphi\in W^{1,\infty}(Y;\RR^m):\varphi=0\hbox{ on }\partial Y\}$ and $\Aff_0(Y;\RR^m):=\{\varphi\in \Aff(Y;\RR^m):\varphi=0\hbox{ on }\partial Y\}$ with $\Aff(Y;\RR^m)$ denoting the space of continuous piecewise affine functions from $Y$ to $\RR^m$.
\begin{remark}
\rm One always has $W\geq\Z W\geq \Z_\infty W\geq\Q W$.
\end{remark}
\begin{theorem}[Representation of $\Q W$ \cite{dacorogna82}]\label{DacorognaTheorem1}
If $W$ is continuous and finite then 
$$
\Q W=\Z W=\Z_\infty W.
$$
\end{theorem}

\begin{theorem}[Integral representation of $\overline{I}$ \cite{dacorogna82}]\label{DacorognaTheorem2}
If $W$ is continuous and
$$
\exists c>0\ \ \forall F\in\MM^{m\times N}\ \ W(F)\leq c(1+|F|^p)
$$
then 
$$
\forall\phi\in W^{1,p}(\Omega;\RR^m)\quad \displaystyle\overline{I}(\phi)=\int_\Omega \Q W(\nabla\phi(x))dx.
$$
\end{theorem}

\subsection{Representation of $\Q W$: non-finite case} Theorem \ref{DacorognaTheorem1} can be extended as follows.

\begin{theorem}[\cite{oah-jpm07,oah-jpm08b}]\label{RTnon-finite-case}\ \hskip150mm
\begin{itemize}
\item[(i)] If $\Z_\infty W$ is finite then $\Q W=\Z_\infty W$.
\item[(ii)] If $\Z W$ is finite then $\Q W=\Z W=\Z_\infty W$.

\end{itemize}
\end{theorem}

\begin{proof}
We need (the two last assertions, the first one being used at the end of \S 2.4, of)  the following result.
\begin{lemma}[\cite{fonseca88}]\label{FonsecaTheorem}\ \hskip150mm
\begin{itemize}
\item[(a)] If $\Z_{\infty} W$ (resp. $\Z W$)  is finite then $\Z_{\infty} W$ (resp. $\Z W$)  is rank-one convex. 
\item[(b)] If $\Z_{\infty} W$ (resp. $\Z W$) is finite then $\Z_{\infty} W$ (resp. $\Z W$) is continuous.
\item[(c)] $\Z_\infty W\leq \Z\Z_\infty W$ and $\Z \Z W=\Z W$.
\end{itemize}
\end{lemma}
One always has $W\geq\Z W\geq \Z_\infty W\geq\Q W$. Hence:
\begin{itemize}
\item[(j)] $\Q\Z_\infty W=\Q W\leq\Z_\infty W$;
\item[(jj)] $\Q\Z W=\Q\Z_\infty W=\Q W$.
\end{itemize}

(i) If $\Z_\infty W$ is finite then $\Z_\infty W$ is continuous by Lemma \ref{FonsecaTheorem}(b). From Theorem \ref{DacorognaTheorem1} it follows that $\Q\Z_\infty W=\Z\Z_\infty W$. But $\Z_\infty W\leq\Z\Z_\infty W$ by Lemma \ref{FonsecaTheorem}(c), and so $\Q W=\Z_\infty W$ by using (j).

\medskip

(ii) If $\Z W$ is finite then also is $\Z_\infty W$. Hence $\Q W=\Z_\infty W$ by the previous reasoning. On the other hand, $\Z W$ is continuous by Lemma \ref{FonsecaTheorem}(b). From Theorem \ref{DacorognaTheorem1} it follows that $\Q\Z W=\Z\Z W$. But $\Z\Z W=\Z W$ by Lemma \ref{FonsecaTheorem}(c), and so $\Q W=\Z W$ by using (jj).
\end{proof}

\begin{question}
Prove (or disprove) that if $\Z_\infty W$ is finite, also is $\Z W$.
\end{question}

\subsection{Representation of $\overline{I}$: non-finite case} Theorem \ref{DacorognaTheorem2} can be extended as follows.

\begin{theorem}[\cite{oah-jpm07,oah-jpm08b}]\label{IRTnon-finite-case}\ \hskip150mm
\begin{itemize}
\item[(i)] If 
$
\exists c>0\ \ \forall F\in\MM^{m\times N}\ \ \Z_\infty W(F)\leq c(1+|F|^p)
$
then
$$
\forall \phi\in W^{1,p}(\Omega;\RR^m)\ \ \displaystyle\overline{I}(\phi)=\int_\Omega \Q W(\nabla\phi(x))dx.
$$
\item[(ii)] If 
$
\exists c>0\ \ \forall F\in\MM^{m\times N}\ \ \Z  W(F)\leq c(1+|F|^p)
$
then
$$
\forall \phi\in W^{1,p}(\Omega;\RR^m)\ \ \displaystyle\overline{I}(\phi)=\overline{I}_{\rm aff}(\phi)=\int_\Omega \Q W(\nabla\phi(x))dx
$$
with $\overline{I}_{\rm aff}:W^{1,p}(\Omega;\RR^m)\to[0,+\infty]$ defined by
$$
\overline{I}_{\rm aff}(\phi):=\inf\left\{\liminf_{n\to+\infty}I(\phi_n):\Aff(\Omega;\RR^m)\ni\phi_n\stackrel{L^p}{\to}\phi\right\}.
$$
\end{itemize}
\end{theorem}
\begin{proof} (i) Let $\Z_\infty I, \overline{\Z_\infty I}, \overline{\Z_\infty I}_{\rm aff}:W^{1,p}(\Omega;\RR^m)\to[0,+\infty]$ be respectively defined by:
\begin{itemize}
\item[\SMALL$\blacklozenge$] $\displaystyle\Z_\infty I(\phi):=\int_\Omega\Z_\infty W(\nabla\phi(x))dx$;
\item[\SMALL$\blacklozenge$] $\displaystyle\overline{\Z_\infty I}(\phi):=\inf\left\{\liminf_{n\to+\infty}\Z_\infty I(\phi_n):\phi_n\stackrel{L^p}{\to}\phi\right\}$;
\item[\SMALL$\blacklozenge$] $\displaystyle\overline{\Z_\infty I}_{\rm aff}(\phi):=\inf\left\{\liminf_{n\to+\infty}\Z_\infty I(\phi_n):\Aff(\Omega;\RR^m)\ni\phi_n\stackrel{L^p}{\to}\phi\right\}.$ 
\end{itemize}
Since $\Z_\infty W$ is of $p$-polynomial growth, i.e., $\exists c>0\ \forall F\in\MM^{m\times N}\ \Z_\infty W(F)\leq c(1+|F|^p)$, it follows that $\Z_\infty W$ is (finite and so) continuous by Lemma \ref{FonsecaTheorem}(b). By Theorem \ref{DacorognaTheorem2} we deduce that
$$
\forall\phi\in W^{1,p}(\Omega;\RR^m)\ \ \overline{\Z_\infty I}(\phi)=\int_\Omega \Q\Z_\infty W(\nabla\phi(x))dx.
$$
But one always has $\Q\Z_\infty W=\Q W$, hence
$$
\forall\phi\in W^{1,p}(\Omega;\RR^m)\ \ \overline{\Z_\infty I}(\phi)=\int_\Omega \Q W(\nabla\phi(x))dx.
$$
Thus, it suffices to prove that $\overline{I}\leq\overline{\Z_\infty I}$ (the reverse inequality being trivially true). The key point of the proof is that we can establish (by using the Vitali covering theorem and without assuming that $\Z_\infty W$ is of $p$-polynomial growth) the following lemma (whose proof is given in \S 2.8.1).

\begin{lemma}\label{R-Lemma1}
$\overline{I}\leq\overline{\Z_\infty I}_{\rm aff}$.
\end{lemma}
On the other hand, as $\Z_\infty W$ is of $p$-polynomial growth and $\Aff(\Omega;\RR^m)$ is strongly dense in $W^{1,p}(\Omega;\RR^m)$, it is easy to see that $\overline{\Z_\infty I}_{\rm aff}=\overline{\Z_\infty I}$, and the result follows.

\medskip

(ii) Let $\Z I, \overline{\Z_\infty I}, \overline{\Z_\infty I}_{\rm aff}:W^{1,p}(\Omega;\RR^m)\to[0,+\infty]$ be respectively defined by:
\begin{itemize}
\item[\SMALL$\blacklozenge$] $\displaystyle\Z I(\phi):=\int_\Omega\Z W(\nabla\phi(x))dx$;
\item[\SMALL$\blacklozenge$] $\displaystyle\overline{\Z I}(\phi):=\inf\left\{\liminf_{n\to+\infty}\Z I(\phi_n):\phi_n\stackrel{L^p}{\to}\phi\right\}$;
\item[\SMALL$\blacklozenge$] $\displaystyle\overline{\Z I}_{\rm aff}(\phi):=\inf\left\{\liminf_{n\to+\infty}\Z I(\phi_n):\Aff(\Omega;\RR^m)\ni\phi_n\stackrel{L^p}{\to}\phi\right\}.$ 
\end{itemize}
As $\Z W$ is of $p$-polynomial growth and (so) continuous (by Lemma \ref{FonsecaTheorem}(b)), from Theorem \ref{DacorognaTheorem2} (and since $\Q\Z W=\Q W$ is always true) we deduce that
$$
\forall\phi\in W^{1,p}(\Omega;\RR^m)\ \ \overline{\Z I}(\phi)=\int_\Omega \Q W(\nabla\phi(x))dx.
$$
It is then sufficient to prove that $\overline{I}_{\rm aff}\leq\overline{\Z I}$ (the inequalities $\overline{I}\leq\overline{I}_{\rm aff}$ and $\overline{\Z I}\leq\overline{I}$ being trivially true). As above, the key point of the proof is that we can establish (by using the Vitali covering theorem and without assuming that $\Z W$ is of $p$-polynomial growth) the following lemma (whose proof is given in \S 2.8.1).
\begin{lemma}\label{R-Lemma2}
$\overline{I}_{\rm aff}\leq\overline{\Z I}_{\rm aff}$.
\end{lemma}
On the other hand, as $\Z W$ is of $p$-polynomial growth and $\Aff(\Omega;\RR^m)$ is strongly dense in $W^{1,p}(\Omega;\RR^m)$, it is clear that $\overline{\Z I}_{\rm aff}=\overline{\Z I}$, and the result follows.
\end{proof}

\medskip

We see here that the integrands $W$ for which $\Z_\infty W$ or $\Z W$ is of $p$-polynomial have a ``nice" behavior with respect to Problem \ref{problem1}. So, it could be interesting to introduce a new class of integrands (that we will call the class of $p$-ample integrands) as follows.
\begin{definition}\label{Ample-Integrands}
We say that $W$ is $p$-ample if and only if $\Z_\infty W$ is of $p$-polynomial growth, i.e., $\exists c>0\ \ \forall F\in\MM^{m\times N}\ \ \Z_\infty W(F)\leq c(1+|F|^p)$.
\end{definition}
We use the term ``$p$-ample" because of some analogies with the concept (developed in differential geometry by Gromov) of amplitude of a differential relation (see \cite{gromov86} for more details). Thus, Theorems \ref{RTnon-finite-case} and \ref{IRTnon-finite-case} can be summarized as follows. 

\begin{theorem}\label{TheoremA-B}
If $W$ is $p$-ample then 
$$
\forall\phi\in W^{1,p}(\Omega;\RR^m)\ \ \overline{I}(\phi)=\int_\Omega\Q W(\nabla\phi(x))dx\hbox{ and $\Q W=\Z_\infty W$.}
$$ 
\end{theorem}
\begin{question}
Prove (or disprove) that $W$ is $p$-ample if and only if $\Q W$ is of $p$-polynomial growth.
\end{question}

An analogue result of Theorem \ref{IRTnon-finite-case} was proved by Ben Belgacem (who is in fact the first that obtained an integral representation for $\overline{I}$ in the non-finite case). Let  $\{\R_i W\}_{i\in\NN}$ be defined by $\R_0 W:=W$ and for each $i\in\NN^*$ and each $F\in\MM^{m\times N}$,
$$
\R_{i+1}W(F):=\infff\limits_{t\in[0,1]}\big\{(1-t)\R_i W(F-t a\otimes b)+t\R_i W(F+(1-t)a\otimes b)\big\}.
$$
By Kohn and Strang (see \cite{kohn-strang86}) we have $\R_{i+1}W\leq\R_i W$ for all $i\in\NN$ and $\R W=\inf_{i\geq 0}\R_i W$, where $\R W$ denotes the rank-one convex envelope of $W$. The Ben Belgacem theorem can be stated as follows.

\begin{theorem}[\cite{benbelgacem96,benbelgacem00}] \label{BenBelgacemTheorem} Assume that{\rm:}
\begin{itemize}
\item[(BB$_1$)] $\OO_W:={\rm int}\Big\{F\in\MM^{m\times N}:\forall i\in\NN\ \Z\R_i W(F)\leq\R_{i+1}W(F)\Big\}$ is dense in $\MM^{m\times N};$
\item[(BB$_2$)] $\forall i\in\NN^*\ $ $\forall F\in\MM^{m\times N}\ $ $\forall\{F_n\}_n\subset \OO_W$ 
$$F_n\to F\then \R_i W(F)\geq\limsup\limits_{n\to+\infty}\R_i W(F_n);$$
\item[\SMALL$\blacklozenge$] $\exists c>0\ \ \forall F\in\MM^{m\times N}\ \ \R W(F)\leq c(1+|F|^p)$.
\end{itemize}
Then
$$
\forall\phi\in W^{1,p}(\Omega;\RR^m)\quad \displaystyle\overline{I}(\phi)=\int_\Omega \Q\R W(\nabla\phi(x))dx.
$$
\end{theorem}

Generally speaking, as rank-one convexity and quasiconvexity do not coincide, Theorem \ref{IRTnon-finite-case} and Theorem \ref{BenBelgacemTheorem} are not identical. However, we have

\begin{lemma} 
If either $\Z_\infty W$ or $\Z W$ is finite then $\Q\R W=\Q W$.
\end{lemma}

\begin{proof}
If $\Z_\infty W$ (resp. $\Z W$) is finite then $\Z_\infty W$ (resp. $\Z W$) is rank-one convex by Lemma \ref{FonsecaTheorem}(a). Consequently $\Z_\infty W\leq\R W$ (resp. $\Z W\leq\R W$) (and Theorem \ref{IRTnon-finite-case-bis} below follows by applying Theorem \ref{IRTnon-finite-case}). Thus, we have $\Z_\infty W\leq\R W\leq W$ (resp. $\Z_\infty W\leq\R W\leq W$), hence $\Q\Z_\infty W\leq\Q\R W\leq \Q W$ (resp. $\Q\Z W\leq\Q\R W\leq \Q W$) and so $\Q\R W=\Q W$ since one always has $\Q\Z_\infty W=\Q W$ (resp. $\Q\Z W=\Q W$).
\end{proof}

\begin{theorem}\label{IRTnon-finite-case-bis}
Assume that $\exists c>0\ \forall F\in\MM^{m\times N}\ \R W(F)\leq c(1+|F|^p)$. Then{\rm :}
\begin{itemize}
\item[(i)] if $\Z_\infty W$ is finite then 
$$\displaystyle\forall \phi\in W^{1,p}(\Omega;\RR^m)\ \ \displaystyle\overline{I}(\phi)=\int_\Omega \Q W(\nabla\phi(x))dx;
$$ 
\item[(ii)] if $\Z W$ is finite then 
$$
\forall \phi\in W^{1,p}(\Omega;\RR^m)\ \ \displaystyle\overline{I}(\phi)=\overline{I}_{\rm aff}(\phi)=\int_\Omega \Q W(\nabla\phi(x))dx.
$$
\end{itemize}
\end{theorem}

\begin{question}
Prove (or disprove) that if {\rm(BB$_1$)} and {\rm(BB$_2$)} hold then $\Z W$ is finite.
\end{question}

\subsection{Application 1: ``non-zero-Cross Product Constraint"} Consider $W_0:\MM^{3\times 2}\to[0,+\infty]$ Borel measurable and $p$-coercive and the following condition 
\begin{equation}\label{P}
\exists\alpha,\beta>0\ \forall \xi=(\xi_1\mid \xi_2)\in\MM^{3\times 2}\ \big(|\xi_1\land \xi_2|\geq\alpha\then W_0(\xi)\leq\beta(1+|\xi|^p)\big)
\end{equation}
with $\xi_1\land\xi_2$ denoting the cross product of vectors $\xi_1,\xi_2\in\RR^3$. When $W_0$ satisfies (\ref{P}) it is compatible with the ``non-zero-Cross Product Constraint", i.e., with the following two conditions:
\begin{equation}\label{*-CPC}
\left\{\begin{array}{l}W_0(\xi_1\mid\xi_2)=+\infty\iff |\xi_1\land\xi_2|=0\\
W_0(\xi_1\mid\xi_2)\to+\infty\hbox{ as }|\xi_2\land\xi_2|\to 0.\end{array}\right.
\end{equation}
 The interest of considering (\ref{*-CPC}) comes from the 3d-2d problem (see \S 3): if $W$ is compatible with the ``strong-Determinant Constraint", i.e., (\ref{s-DC}), then $W_0$ given by $W_0(\xi):=\inf_{\zeta\in\RR^3} W(\xi\mid\zeta)$ is compatible with (\ref{*-CPC}). One can establish the following lemma (whose proof is given in \S 2.8.2) which roughly means that the ``non-zero Cross Product Constraint" is $p$-ample.
 \begin{lemma}[\cite{oah-jpm07,oah-jpm08a}]\label{R-Lemma3}
 If $W_0$ satisfies {\rm (\ref{P})} then $\Z W_0$ is of $p$-polynomial growth, i.e., $\exists c>0\ \forall \xi\in\MM^{3\times 2}\ \Z W_0(\xi)\leq c(1+|\xi|^p)$.
\end{lemma}
Applying Theorem \ref{IRTnon-finite-case}(ii) we obtain

\begin{corollary}\label{corollary1}
If $W_0$ satisfies {\rm(\ref{P})} then
$$
\forall \psi\in W^{1,p}(\Omega;\RR^3)\ \ \displaystyle\overline{I}(\psi)=\overline{I}_{\rm aff}(\psi)=\int_\Omega \Q W_0(\nabla\psi(x))dx.
$$
\end{corollary}

\subsection{Application 2: ``weak-Determinant Constraint"} The following condition on $W$ is compatible with ``weak-Determinant Constraint", i.e., (\ref{w-DC}).
\begin{equation}\label{D}
\exists\alpha,\beta>0\ \forall F\in\MM^{N\times N}\ \big(|\det F|\geq\alpha\then W(F)\leq\beta(1+|F|^p)\big).
\end{equation}
One can prove the following lemma which roughly means that the ``weak-Determinant Constraint" is $p$-ample (see also Lemma \ref{LeMMa-wDC}-bis below).
\begin{lemma}[\cite{oah-jpm08a}]\label{LeMMa-wDC}
If $W$ satisfies {\rm(\ref{D})} then $\Z W$ is of $p$-polynomial growth, i.e.,  $\exists c>0\ \forall F\in\MM^{N\times N}\ \Z W(F)\leq c(1+|F|^p)$.
\end{lemma}
Applying Theorem \ref{IRTnon-finite-case}(ii) we obtain

\begin{corollary}\label{corollary2}
If $W$ satisfies {\rm(\ref{D})} then
$$
\forall \phi\in W^{1,p}(\Omega;\RR^N)\ \ \displaystyle\overline{I}(\phi)=\overline{I}_{\rm aff}(\phi)=\int_\Omega \Q W(\nabla\phi(x))dx.
$$
\end{corollary}

\begin{proof}[\bf Proof (of a part of Corollary \ref{corollary2})]
Taking Theorem \ref{IRTnon-finite-case-bis}(i) into account, it suffices to verify the following two points:
\begin{itemize}
\item[\SMALL$\blacklozenge$] (\ref{D}) $\then$ $\exists c>0\ \forall F\in\MM^{N\times N}\ \R W(F)\leq c(1+|F|^p)$;
\item[\SMALL$\blacklozenge$] (\ref{D}) $\then$ $\Z_\infty W<+\infty$,
\end{itemize}
which will give us the desired integral representation for $\overline{I}$. The first point is essentially due to a lemma by Ben Belgacem: it is a direct consequence of Theorem \ref{SingularValues} (see Remark \ref{RemarkSingularValues}) whose proof is given in \S 3.6.4.  For the second point, it is obvious that  $\Z_\infty W(F)<+\infty$ for all $F\in\MM^{N\times N}$ with $|\det F|\geq \alpha$. On the other hand, we have

\begin{lemma}[\cite{daco-rib04}, see also \cite{celada-perrotta98}]\label{LeMMa}
For all $F\in\MM^{N\times N}$, if $|\det F|<\alpha$ then there exists 
 $\varphi\in W^{1,\infty}_0(Y;\RR^N)$ such that $|\det(F+\nabla\varphi(x))|=\alpha$ a.e. in $Y$.
\end{lemma}
Hence, if $F\in\MM^{N\times N}$ is such that $|\det F|<\alpha$ then $\Z_\infty W(F)\leq\int_{Y}W(F+\nabla\varphi(x))dx$ with some $\varphi\in W^{1,\infty}_0(Y;\RR^N)$ given by Lemma \ref{LeMMa}, and so $\Z_\infty W(F)\leq 2^p\beta(1+|F|^p+\|\nabla\varphi\|^p_{L^p})<+\infty$.
\end{proof}

\begin{remark}
From the previous proof, we can isolate the following result.
\newtheorem*{LemmaBis}{\bf Lemma \ref{LeMMa-wDC}-bis}
\begin{LemmaBis}
If $W$ satisfies {\rm(\ref{D})} then $\Z_\infty W$ is of $p$-polynomial growth, i.e.,  $\exists c>0\ \forall F\in\MM^{N\times N}\ \Z_\infty W(F)\leq c(1+|F|^p)$.
\end{LemmaBis}
\end{remark}

\subsection{From $p$-ample to non-$p$-ample case}

Because of the following theorem, none of the theorems of this section  can be directly used for dealing with Problem \ref{problem1} under the ``stong-Determinant Constraint", i.e., (\ref{s-DC}).
\begin{theorem}[\cite{fonseca88}]
If $W$ satisfies {\rm(\ref{s-DC})} then{\rm:}
\begin{itemize}
\item[(F$_1$)] $\Q W$ is rank-one convex{\rm;}
\item[(F$_2$)] $\Q W(F)=+\infty$ if and only if $\det F\leq 0$ and $\Q W(F)\to+\infty$ as $\det F\to 0^+$.
\end{itemize}
\end{theorem}
The assertion (F$_2$) roughly says  that the ``strong-Determinant Constraint" is not $p$-ample, i.e., $\Z_\infty W$ cannot be of $p$-polynomial growth, and so neither Theorem \ref{RTnon-finite-case} nor Theorem \ref{IRTnon-finite-case} is consistent with (\ref{s-DC}). From the assertion (F$_1$) we see that $\Q W\leq\R W$ which shows that $\R W$ cannot be of $p$-polynomial growth when combined with (F$_2$). Hence, the theorem of Ben Belgacem is not compatible with (\ref{s-DC}).

\begin{question}
Develop strategies for passing from $p$-ample to non-$p$-ample case.
\end{question}

\subsection{Complementary proofs}

\subsubsection{Proof of Lemmas \ref{R-Lemma1} and \ref{R-Lemma2}} It is sufficient to prove that if $\phi\in\Aff(\Omega;\RR^m)$ then
\begin{eqnarray}\label{step1VitAli}
\overline{I}(\phi)\leq\int_\Omega \Z_\infty W(\nabla \phi(x))dx\hbox{ (resp. }\overline{I}_{\rm aff}(\phi)\leq \int_\Omega \Z W(\nabla \phi(x))dx).
\end{eqnarray}
By definition, there exists a finite family $(V_i)_{i\in I}$ of open disjoint subsets of $\Omega$ such that $|\Omega\setminus\cup_{i\in I}V_i|=0$ and for every $i\in I$, $|\partial V_i|=0$ and $\nabla\phi(x)=F_i$ in $V_i$ with $F_i\in\MM^{m\times N}$.  Given $\delta>0$ and $i\in I$, we consider $\varphi_i\in W^{1,\infty}_0(Y;\RR^m)$ (resp. $\varphi_i\in \Aff_0(Y;\RR^m)$) such that
\begin{eqnarray}\label{Zk}
&&\int_Y W(F_i+\nabla\varphi_i(y))dy\leq\Z_\infty W(F_i)+{\delta\over|\Omega|}\\ 
&&\hbox{(resp. }\int_Y W(F_i+\nabla\varphi_i(y))dy\leq\Z W(F_i)+{\delta\over|\Omega|}).\nonumber
\end{eqnarray}
Fix any integer $n\geq 1$. By the Vitali covering theorem, there exists a finite or countable family $(a_{i,j}+\alpha_{i,j}Y)_{j\in J_{i}}$ of disjoint subsets of $V_i$, where $a_{i,j}\in\RR^N$ et $0<\alpha_{i,j}<{1\over n}$, such that
$
\big|V_i\setminus\cup_{j\in J_{i}}(a_{i,j}+\alpha_{i,j}Y)\big|=0
$
 (and so $\sum_{j\in J_i}\alpha_{i,j}^N=|V_i|$). Define $\phi_n\in W^{1,\infty}_0(\Omega;\RR^m)$ (resp. $\phi_n\in \Aff_0(\Omega;\RR^m)$) by 
$$
\phi_n(x):=
\alpha_{i,j}\varphi_{i}\left({x-a_{i,j}\over \alpha_{i,j}}\right)\hbox{ si }x\in a_{i,j}+\alpha_{i,j}Y.
$$
Clearly, we have  $\|\phi_n\|_{L^\infty(\Omega;\RR^m)}\leq {1\over n}\max_{i\in I}\|\varphi_i\|_{L^\infty(Y;\RR^m)}$ and $\|\nabla\phi_n\|_{L^\infty(\Omega;\MM^{m\times N})}\leq\max_{i\in I}\|\nabla\varphi_i\|_{L^\infty(Y;\MM^{m\times N})}$, and so, up to a subsequence,  $\phi_n\stackrel{*}{\wto}0$ in $W^{1,\infty}(\Omega;\RR^m)$, where ``$\stackrel{*}{\wto}$" denotes the weak* convergence in  $W^{1,\infty}(\Omega;\RR^m)$. Hence $\phi_n\wto 0$ in $W^{1,p}(\Omega;\RR^m)$. Consequently, up to a subsequence, $\phi_n\to 0$ in $L^p(\Omega;\RR^m)$. Moreover, we have 
\begin{eqnarray*}
\int_\Omega W\left(\nabla \phi(x)+\nabla\phi_n(x)\right)dx&=&\sum_{i\in I}\int_{V_i} W\left(F_i+\nabla\phi_n(x)\right)dx\\
&=&\sum_{i\in I}\sum_{j\in J_i}\alpha_{i,j}^N\int_{Y}W\left(F_i+\nabla\varphi_{i}(y)\right)dy\\
&=&\sum_{i\in I}|V_i|\int_{Y}W\left(F_i+\nabla\varphi_{i}(y)\right)dy.
\end{eqnarray*}
Since $\phi+\phi_n\in W^{1,\infty}(\Omega;\RR^m)$ (resp. $\phi+\phi_n\in\Aff(\Omega;\RR^m)$) and $\phi+\phi_n\to \phi$ in $L^{p}(\Omega;\RR^m)$, using (\ref{Zk}) we deduce that 
\begin{eqnarray*}
\overline{I}(\phi)\leq\liminf_{n\to+\infty}\int_\Omega W\left(\nabla \phi(x)+\nabla\phi_n(x)\right)dx&\leq&\sum_{i\in I}|V_i|\Z W(F_i)+\delta\\
&=&\int_\Omega\Z_\infty W(\nabla \phi(x))dx + \delta\\
\hbox{(resp. }\overline{I}_{\rm aff}(\phi)\leq \int_\Omega\Z W(\nabla \phi(x))dx + \delta),
\end{eqnarray*}
and (\ref{step1VitAli}) follows by letting $\delta\to 0$. $\blacksquare$


\subsubsection{Proof of Lemma \ref{R-Lemma3}} We begin by proving that $\Z W_0$ satisfies the following condition.
\begin{equation}\label{InTERConditioN}
\exists\gamma>0\ \forall \xi\in\MM^{3\times 2}\ \big(\min\{|\xi_1+\xi_2|,|\xi_1-\xi_2|\}\geq\alpha\then W_0(F)\leq\gamma(1+|F|^p)\big).
\end{equation}

Let $\xi=(\xi_1\mid \xi_2)\in\MM^{3\times 2}$ be such that $\min\{|\xi_1+\xi_2|,|\xi_1-\xi_2|\}\geq\alpha$. Then, one of the three possibilities holds:
\begin{eqnarray}\label{Possibility1Vect}
&&|\xi_1\land \xi_2|\not=0;\\ 
&&|\xi_1\land \xi_2|=0\hbox{ avec }\xi_1\not=0;\label{Possibility2Vect}\\
&&|\xi_1\land \xi_2|=0\hbox{ avec }\xi_2\not=0. \label{Possibility3Vect}
\end{eqnarray}
Set
$
D:=\{(x_1,x_2)\in\RR^2:x_1-1<x_2<x_1+1\hbox{ et }-x_1-1<x_2<1-x_1\}
$ 
and define $\psi\in\Aff_0(D;\RR)$ by
$$
\psi(x_1,x_2):=\left\{
\begin{array}{ll}
-x_1+(x_2+1)&\hbox{si }(x_1,x_2)\in\Delta_1\\
(1-x_1)-x_2&\hbox{si }(x_1,x_2)\in\Delta_2\\
x_1+(1-x_2)&\hbox{si }(x_1,x_2)\in\Delta_3\\
(x_1+1)+x_2&\hbox{si }(x_1,x_2)\in\Delta_4
\end{array}
\right.
$$
with:
\begin{itemize} 
\item[]$\Delta_1:=\{(x_1,x_2)\in D:x_1\geq 0\hbox{ et } x_2\leq 0\}$;
\item[]$\Delta_2:=\{(x_1,x_2)\in D:x_1\geq 0\hbox{ et }x_2\geq 0\}$;
\item[]$\Delta_3:=\{(x_1,x_2)\in D:x_1\leq 0\hbox{ et }x_2\geq 0\}$;
\item[]$\Delta_4:=\{(x_1,x_2)\in D:x_1\leq 0\hbox{ et }x_2\leq 0\}$.
\end{itemize}
Consider $\varphi\in\Aff_0(D;\RR^3)$ given by
$$
\varphi(x):=\psi(x)\nu\hbox{ avec }\left\{\begin{array}{ll}\nu={\xi_1\land \xi_2\over|\xi_1\land \xi_2|}&\hbox{si on a (\ref{Possibility1Vect})}\\
|\nu|=1 \hbox{ et }\langle \xi_1,\nu\rangle=0&\hbox{si on a (\ref{Possibility2Vect})}\\
|\nu|=1 \hbox{ et }\langle \xi_2,\nu\rangle=0&\hbox{si on a (\ref{Possibility3Vect}),}\end{array}\right.
$$
where $\langle\cdot,\cdot\rangle$ denotes the scalar product in $\RR^3$. Then
$$
\xi+\nabla\varphi(x)=\left\{
\begin{array}{ll}
(\xi_1-\nu\mid \xi_2+\nu)&\hbox{si }x\in{\rm int}(\Delta_1)\\
(\xi_1-\nu\mid \xi_2-\nu)&\hbox{si }x\in{\rm int}(\Delta_2)\\
(\xi_1+\nu\mid \xi_2-\nu)&\hbox{si }x\in{\rm int}(\Delta_3)\\
(\xi_1+\nu\mid \xi_2+\nu)&\hbox{si }x\in{\rm int}(\Delta_4)
\end{array}
\right.
$$
with ${\rm int}(E)$ denoting the interior of $E$. We need the following result.
\begin{lemma}[\cite{fonseca88}]\label{FonsecaLemma}
For every bounded open set $D\subset\RR^2$ with $|\partial D|=0$ and every $\xi\in\MM^{3\times 2}$,
$$
\Z W_0(\xi)=\inf\left\{{1\over |D|}\int_D W_0(\xi+\nabla\varphi(x))dx:\varphi\in\Aff_0(D;\RR^3)\right\}.
$$
\end{lemma}
Using Lemma \ref{FonsecaLemma} we deduce that
\begin{eqnarray}\label{Z_1}
\Z W_0(\xi)&\leq&{1\over 4}\left(W_0(\xi_1-\nu\mid \xi_2+\nu)+W_0(\xi_1-\nu\mid \xi_2-\nu)\right.\\
&&\left.+\ W_0(\xi_1+\nu\mid \xi_2-\nu)+W_0(\xi_1+\nu\mid \xi_2+\nu)\right).\nonumber
\end{eqnarray}
But
$
|(\xi_1-\nu)\land(\xi_2+\nu)|^2=|\xi_1\land \xi_2+(\xi_1+\xi_2)\land\nu|^2
=|\xi_1\land \xi_2|^2+|(\xi_1+\xi_2)\land\nu|^2
\geq|(\xi_1+\xi_2)\land\nu|^2,
$
hence
$$
|(\xi_1+\nu)\land(\xi_2-\nu)|\geq |(\xi_1+\xi_2)\land\nu|=|\xi_1+\xi_2|.
$$
Similarly, we obtain: 
\begin{itemize}
\item[]$|(\xi_1-\nu)\land(\xi_2-\nu)|\geq |\xi_1-\xi_2|$;
\item[]$|(\xi_1+\nu)\land(\xi_2-\nu)|\geq |\xi_1+\xi_2|$;
\item[]$|(\xi_1+\nu)\land(\xi_2+\nu)|\geq |\xi_1-\xi_2|$.
\end{itemize}
Thus $|(\xi_1-\nu)\land(\xi_2+\nu)|\geq\alpha$, $|(\xi_1-\nu)\land(\xi_2-\nu)|\geq\alpha$, $|(\xi_1+\nu)\land(\xi_2-\nu)|\geq\alpha$ et $|(\xi_1+\nu)\land(\xi_2+\nu)|\geq\alpha$ because $\min\{|\xi_1+\xi_2|,|\xi_1-\xi_2|\}\geq\alpha$. Using {\rm(\ref{P})} it follows that
\begin{eqnarray*}
W_0(\xi_1-\nu\mid \xi_2+\nu)&\leq&\beta(1+|(\xi_1-\nu\mid \xi_2+\nu)|^p)\\
&\leq&\beta 2^p(1+|(\xi_1\mid \xi_2)|^p+|(-\nu\mid \nu)|^p)\\
&\leq&\beta 2^{2p+1}(1+|\xi|^p).
\end{eqnarray*}
In the same manner, we have: 
\begin{itemize}
\item[]$W_0(\xi_1-\nu\mid \xi_2-\nu)\leq \beta 2^{2p+1}(1+|\xi|^p)$;
\item[]$W_0(\xi_1+\nu\mid \xi_2-\nu)\leq \beta 2^{2p+1}(1+|\xi|^p)$;
\item[]$W_0(\xi_1+\nu\mid \xi_2+\nu)\leq \beta 2^{2p+1}(1+|\xi|^p)$,
\end{itemize}
and from (\ref{Z_1}) we conclude that $\Z W_0(\xi)\leq \beta 2^{2p+1}(1+|\xi|^p)$, which proves (\ref{InTERConditioN}).

\medskip

We now prove that $\Z W_0$ is of $p$-polynomial growth, i.e., 
\begin{equation}\label{EndConDITIoN}
\exists c>0\ \forall F\in\MM^{3\times 2}\ \Z W_0(\xi)\leq c(1+|\xi|^p).
\end{equation}
Let $\xi=(\xi_1\mid \xi_2)\in\MM^{3\times 2}$. Then, one of the four possibilities holds: 
\begin{eqnarray}
&&|\xi_1\land \xi_2|\not=0;\label{Possibility(i)Vect}\\
&&|\xi_1\land \xi_2|=0\hbox{ avec }\xi_1=\xi_2=0;\label{Possibility(ii)Vect}\\
&&|\xi_1\land \xi_2|=0\hbox{ avec }\xi_1\not=0;\label{Possibility(iii)Vect}\\
&&|\xi_1\land \xi_2|=0\hbox{ avec }\xi_2\not=0. \label{Possibility(iv)Vect}
\end{eqnarray}
Define $\psi\in\Aff_0(Y;\RR)$ by 
$$
\psi(x_1,x_2):=\left\{
\begin{array}{ll}
x_2&\hbox{si }(x_1,x_2)\in\Delta_1\\
(1-x_1)&\hbox{si }(x_1,x_2)\in\Delta_2\\
(1-x_2)&\hbox{si }(x_1,x_2)\in\Delta_3\\
x_1&\hbox{si }(x_1,x_2)\in\Delta_4
\end{array}
\right.
$$
with:
\begin{itemize} 
\item[]$\Delta_1:=\{(x_1,x_2)\in Y:x_2\leq x_1\leq -x_2+1\}$;
\item[]$\Delta_2:=\{(x_1,x_2)\in Y:-x_1+1\leq x_2\leq x_1\}$;
\item[]$\Delta_3:=\{(x_1,x_2)\in Y:-x_2+1\leq x_1\leq x_2\}$;
\item[]$\Delta_4:=\{(x_1,x_2)\in Y:x_1\leq x_2\leq -x_1+1\}$.
\end{itemize}
Consider $\varphi\in\Aff_0(Y;\RR^3)$ given by 
$$
\varphi(x):=\psi(x)\nu\hbox{ avec }
\left\{
\begin{array}{ll}
\nu={\alpha(\xi_1\land \xi_2)\over|\xi_1\land \xi_2|}&\hbox{si on a (\ref{Possibility(i)Vect})}\\  
|\nu|=\alpha&\hbox{si on a (\ref{Possibility(ii)Vect})}\\
|\nu|=\alpha\hbox{ et }\langle \xi_1,\nu\rangle=0& \hbox{si on a (\ref{Possibility(iii)Vect})}\\
|\nu|=\alpha\hbox{ et }\langle \xi_2,\nu\rangle=0& \hbox{si on a (\ref{Possibility(iv)Vect}).}
\end{array}
\right.
$$
Then
$$
\xi+\nabla\varphi(x)=\left\{
\begin{array}{ll}
(\xi_1\mid \xi_2+\nu)&\hbox{si }x\in{\rm int}(\Delta_1)\\
(\xi_1-\nu\mid \xi_2)&\hbox{si }x\in{\rm int}(\Delta_2)\\
(\xi_1\mid \xi_2-\nu)&\hbox{si }x\in{\rm int}(\Delta_3)\\
(\xi_1+\nu\mid \xi_2)&\hbox{si }x\in{\rm int}(\Delta_4).
\end{array}
\right.
$$
Using Lemma \ref{FonsecaTheorem}(c) together with Lemma \ref{FonsecaLemma} we deduce that
\begin{eqnarray}\label{ZzZ}
\Z W_0(\xi)&\leq&{1\over 4}\left(\Z W_0(\xi_1\mid \xi_2+\nu)+\Z W_0(\xi_1-\nu\mid \xi_2)\right.\\
&&\left.+\ \Z W_0(\xi_1\mid \xi_2-\nu)+\Z W_0(\xi_1+\nu\mid \xi_2)\right).\nonumber
\end{eqnarray}
But
$
|\xi_1+(\xi_2+\nu)|^2=|(\xi_1+\xi_2)+\nu|^2
=|\xi_1+\xi_2|^2+|\nu|^2
=|\xi_1+\xi_2|^2+\alpha^2
\geq \alpha^2,
$
hence
$
|\xi_1+(\xi_2+\nu)|\geq \alpha.
$
Similarly, we obtain
$
|\xi_1-(\xi_2+\nu)|\geq \alpha,
$
and so
$$
\min\{|\xi_1+(\xi_2+\nu)|,|\xi_1-(\xi_2+\nu)|\}\geq \alpha.
$$
In the same manner, we have: 
\begin{itemize}
\item[]$\min\{|(\xi_1-\nu)+\xi_2|,|(\xi_1-\nu)-\xi_2|\}\geq\alpha$;
\item[]$\min\{|\xi_1+(\xi_2-\nu)|,|\xi_1-(\xi_2-\nu)|\}\geq\alpha$;
\item[]$\min\{|(\xi_1+\nu)+\xi_2|,|(\xi_1+\nu)-\xi_2|\}\geq\alpha$. 
\end{itemize}
As $\Z W_0$ satisfies (\ref{InTERConditioN}) it follows that
\begin{eqnarray*}
\Z W_0(\xi_1\mid \xi_2+\nu)&\leq&\gamma(1+|(\xi_1\mid \xi_2+\nu)|^p)\\
&\leq&\gamma 2^p(1+|(\xi_1\mid \xi_2)|^p+|(0\mid\nu)|^p)\\
&\leq&\max\{1,\alpha^p\}\gamma2^{p+1}(1+|\xi|^p).
\end{eqnarray*}
In the same manner, we obtain:
\begin{itemize}
\item[]$\Z W_0(\xi_1-\nu\mid \xi_2)\leq\max\{1,\alpha^p\}\gamma2^{p+1}(1+|\xi|^p)$;
\item[]$\Z W_0(\xi_1\mid \xi_2-\nu)\leq\max\{1,\alpha^p\}\gamma2^{p+1}(1+|\xi|^p)$;
\item[]$\Z W_0(\xi_1+\nu\mid \xi_2)\leq\max\{1,\alpha^p\}\gamma2^{p+1}(1+|\xi|^p)$,
\end{itemize} 
and from (\ref{ZzZ}) we conclude that $\Z W_0(\xi)\leq\max\{1,\alpha^p\}\gamma2^{p+1}(1+|\xi|^p)$, which proves (\ref{EndConDITIoN}). $\blacksquare$


\section{3d-2d passage theorems with determinant type constraints}

\subsection{Statement of the problem} Let $W:\MM^{3\times 3}\to[0,+\infty]$ be Borel measurable and $p$-coercive (with $p>1$) and, for each $\eps>0$, let $I_\eps:W^{1,p}(\Sigma_\eps;\RR^3)\to[0,+\infty]$ be defined by
$$
I_\eps(\phi):={1\over\eps}\int_{\Sigma_\eps}W(\nabla\phi(x,x_3))dxdx_3,
$$
where $\Sigma_\eps:=\Sigma\times]-{\eps\over 2},{\eps\over 2}[\subset\RR^3$ with $\Sigma\subset\RR^2$ Lipschitz, open and bounded, and a point of $\Sigma_\eps$ is denoted by $(x,x_3)$ with $x\in\Sigma$ and $x_3\in]-{\eps\over 2},{\eps\over 2}[$. The problem of 3d-2d passage is the following.
\begin{problem}\label{problem2}
Prove (or disprove) that
$$
\forall\psi\in W^{1,p}(\Sigma;\RR^3)\ \ \left(\Gamma(\pi)\hbox{-}\lim_{\eps\to 0}I_\eps\right)(\psi)=\int_\Sigma W_{\rm mem}(\nabla\psi(x))dx,
$$ 
where the symbol {\rm``$\Gamma(\pi)\hbox{-}\lim_{\eps\to 0}$"} stands for the $\Gamma(\pi)$-limit as $\eps\to 0$ (see Definition {\rm\ref{ABPdef}}), and find a representation formula for $W_{\rm mem}:\MM^{3\times 2}\to[0,+\infty]$.
\end{problem}
At the begining of the nineties, in \cite{ledret-raoult93,ledret-raoult95} Le Dret and Raoult answered to Problem \ref{problem2} in the case where $W$ is ``finite and without singularities" (see \S 3.3). Recently, in \cite{oah-jpm06,oah-jpm08b} we extended the Le Dret-Raoult theorem to the case where $W$ is compatible with the ``weak-Determinant constraint", i.e., (\ref{w-DC}), and the ``strong-Determinant Constraint", i.e., (\ref{s-DC}), as Theorem \ref{RD-w} and Theorem \ref{RD-s} (see \S 3.4 and \S 3.5).

\subsection{The $\Gamma(\pi)$-convergence} The concept of $\Gamma(\pi)$-convergence was introduced Anzellotti, Baldo and Percivale in order to deal with dimension reduction problems in mechanics. Let $\pi=\{\pi_\eps\}_\eps$ be the family of $L^p$-continuous maps $\pi_\eps:W^{1,p}(\Sigma_\eps;\RR^3)\to W^{1,p}(\Sigma;\RR^3)$ defined by
$$
\displaystyle\pi_\eps(\phi):={1\over\eps}\int_{-{\eps\over 2}}^{\eps\over 2}\phi(\cdot,x_3)dx_3.
$$

\begin{definition}[\cite{anzellotti-baldo-percivale94}]\label{ABPdef} We say that $\{I_\eps\}_{\eps}$ $\Gamma(\pi)$-converge to $I_{\rm mem}$ as $\eps$ goes to zero, and we write 
$$
I_{\rm mem}=\Gamma(\pi)\hbox{-}\lim\limits_{\eps\to 0} I_\eps,
$$ 
if and only if 
$$
\forall\psi\in W^{1,p}(\Sigma;\RR^3)\quad\left(\Gamma(\pi)\hbox{-}\liminf\limits_{\eps\to 0} I_\eps\right)(\psi)=\left(\Gamma(\pi)\hbox{-}\limsup\limits_{\eps\to 0} I_\eps\right)(\psi)=I_{\rm mem}(\psi)
$$
with $\Gamma(\pi)\hbox{-}\liminf\limits_{\eps\to 0} I_\eps, \Gamma(\pi)\hbox{-}\limsup\limits_{\eps\to 0} I_\eps:W^{1,p}(\Sigma;\RR^3)\to[0,+\infty]$ respectively given by{\rm:}
\begin{itemize}
\item[\SMALL$\blacklozenge$] $\left(\Gamma(\pi)\hbox{-}\liminf\limits_{\eps\to 0} I_\eps\right)(\psi):=\inf\left\{\liminf\limits_{\eps\to 0}I_\eps(\phi_\eps):\pi_\eps(\phi_\eps)\stackrel{L^p}{\to}\psi\right\};$
\item[\SMALL$\blacklozenge$] $\left(\Gamma(\pi)\hbox{-}\limsup\limits_{\eps\to 0} I_\eps\right)(\psi):=\inf\left\{\limsup\limits_{\eps\to 0}I_\eps(\phi_\eps):\pi_\eps(\phi_\eps)\stackrel{L^p}{\to}\psi\right\}$.
\end{itemize}
\end{definition}
Anzellotti, Baldo and Percivale proved that their concept of $\Gamma(\pi)$-convergence is not far from that of $\Gamma$-convergence introduced by De Giorgi and Franzoni.  For each $\eps>0$, consider $\mathcal{I}_\eps:W^{1,p}(\Sigma;\RR^3)\to[0,+\infty]$ defined by
$$
\mathcal{I}_\eps(\psi):=\inf\Big\{I_\eps(\phi):\pi_\eps(\phi)=\psi\Big\}.
$$
\begin{definition}[\cite{degiorgi-franzoni75,degiorgi75}] We say that $\{\mathcal{I}_\eps\}_{\eps}$ $\Gamma$-converge to $I_{\rm mem}$ as $\eps$ goes to zero, and we write 
$$
I_{\rm mem}=\Gamma\hbox{-}\lim\limits_{\eps\to 0} \mathcal{I}_\eps,
$$
if and only if
$$
\forall\psi\in W^{1,p}(\Sigma;\RR^3)\quad\left(\Gamma\hbox{-}\liminf\limits_{\eps\to 0} \mathcal{I}_\eps\right)(\psi)=\left(\Gamma\hbox{-}\limsup\limits_{\eps\to 0} \mathcal{I}_\eps\right)(\psi)=I_{\rm mem}(\psi)
$$
with $\Gamma\hbox{-}\liminf\limits_{\eps\to 0} \mathcal{I}_\eps, \Gamma\hbox{-}\limsup\limits_{\eps\to 0} \mathcal{I}_\eps:W^{1,p}(\Sigma;\RR^3)\to[0,+\infty]$ respectively given by{\rm:}
\begin{itemize}
\item[\SMALL$\blacklozenge$] $\Gamma\hbox{-}\liminf\limits_{\eps\to 0} \mathcal{I}_\eps(\psi):=\inf\left\{\liminf\limits_{\eps\to 0}\mathcal{I}_\eps(\psi_\eps):\psi_\eps\stackrel{L^p}{\to}\psi\right\};$
\item[\SMALL$\blacklozenge$] $\Gamma\hbox{-}\limsup\limits_{\eps\to 0} \mathcal{I}_\eps(\psi):=\inf\left\{\limsup\limits_{\eps\to 0}\mathcal{I}_\eps(\psi_\eps):\psi_\eps\stackrel{L^p}{\to}\psi\right\}.$
\end{itemize}
\end{definition}
The link between $\Gamma(\pi)$-convergence and $\Gamma$-convergence is given by the following lemma. 
\begin{lemma}[\cite{anzellotti-baldo-percivale94}]\label{ABP}

$
I_{\rm mem}=\Gamma(\pi)\hbox{-}\lim\limits_{\eps\to 0} I_\eps$ if and only if $I_{\rm mem}=\Gamma\hbox{-}\lim\limits_{\eps\to 0} \mathcal{I}_\eps
$.
\end{lemma}


\subsection{$\Gamma(\pi)$-convergence of $I_\eps$: finite case}

Let $W_0:\MM^{3\times 2}\to[0,+\infty]$ be defined by
$$
W_0(\xi):=\inf\limits_{\zeta\in\RR^3}W(\xi\mid\zeta).
$$
\begin{theorem}[\cite{ledret-raoult93,ledret-raoult95}]\label{LeDret-Raoult-Theorem} If $W$ is continuous and 
$$
\exists c>0\ \forall F\in\MM^{3\times 3}\ W(F)\leq c(1+|F|^p)
$$
then
$$
\forall\psi\in W^{1,p}(\Sigma;\RR^3)\ \ \Gamma(\pi)\hbox{-}\lim\limits_{\eps\to 0}I_\eps(\psi)=\int_\Sigma\Q W_0(\nabla\psi(x))dx.
$$
\end{theorem}
Although the Le Dret-Raoult theorem is compatible neither with the ``weak-Deter-minant Constraint", i.e., (\ref{w-DC}) nor with the ``strong  Determinant Constraint", i.e., (\ref{s-DC}), it established a suitable variational framework to deal with dimensional reduction problems: it is the point of departure of many works on the subject.

\subsection{$\Gamma(\pi)$-convergence of $I_\eps$: ``weak-Determinant Constraint"} By using the Le Dret-Raoult theorem, i.e., Theorem \ref{LeDret-Raoult-Theorem}, we can prove the following result.

\begin{theorem}[\cite{oah-jpm06}]\label{RD-w}
If $W$ satisfies {\rm(\ref{D})}, i.e., 
$$
\exists\alpha,\beta>0\ \forall F\in\MM^{3\times 3}\ \big(|\det F|\geq\alpha\then W(F)\leq\beta(1+|F|^p)\big),
$$
then
$$
\forall\psi\in W^{1,p}(\Sigma;\RR^3)\quad\Gamma(\pi)\hbox{-}\displaystyle\lim\limits_{\eps\to 0}I_\eps(\psi)=\int_\Sigma\Q W_0(\nabla\psi(x))dx.
$$
\end{theorem}
\begin{proof}

As the $\Gamma(\pi)$-limit is stable by substituting $I_\eps$ by its relaxed functional $\overline{I}_\eps$, i.e., $\overline{I}_\eps:W^{1,p}(\Sigma_\eps;\RR^3)\to[0,+\infty]$ given by
\begin{eqnarray*}
\overline{I}_\eps(\phi)&:=&\inf\left\{\liminf\limits_{n\to+\infty}I_\eps(\phi_n):\phi_n\stackrel{L^p}{\to}\phi\right\}\\
&=&{1\over\eps}\inf\left\{\liminf\limits_{n\to+\infty}\int_{\Sigma_\eps}W(\nabla\phi_n)dxdx_3:\phi_n\stackrel{L^p}{\to}\phi\right\},
\end{eqnarray*}
it suffices to prove that
\begin{equation}\label{ConCLUSioN}
\forall\psi\in W^{1,p}(\Sigma;\RR^3)\quad\Gamma(\pi)\hbox{-}\displaystyle\lim\limits_{\eps\to 0}\overline{I}_\eps(\psi)=\int_\Sigma\Q W_0(\nabla\psi(x))dx. 
\end{equation}

As $W$ satisfies (\ref{D}) it is $p$-ample (see Definition \ref{Ample-Integrands}), and so by Theorem \ref{TheoremA-B} we have 
\begin{equation}\label{ConCLUSioN1}
\forall\eps>0\ \forall \phi\in W^{1,p}(\Sigma_\eps;\RR^3)\quad\overline{I}_\eps(\phi)={1\over\eps}\int_{\Sigma_\eps}\Q W(\nabla\phi(x,x_3))dxdx_3
\end{equation}
with $\Q W=\Z_\infty W$ (which is of $p$-polynomial growth and so continuous by Lemma \ref{FonsecaTheorem}(b)). Applying the Le Dret-Raoult theorem, i.e., Theorem \ref{LeDret-Raoult-Theorem}, we deduce that
$$
\forall \psi\in W^{1,p}(\Sigma;\RR^3)\quad\Gamma(\pi)\hbox{-}\lim_{\eps\to 0}\overline{I}_\eps(\psi)=\int_{\Sigma}\Q[\Q W]_0(\nabla\psi(x))dx
$$
with $[\Q W]_0:\MM^{3\times 2}\to[0,+\infty]$ given by
$$
[\Q W]_0(\xi):=\inf_{\zeta\in\RR^3}\Q W(\xi\mid\zeta).
$$

On the other hand, one can establish the following lemma (whose proof is given in \S 3.6.1).
\begin{lemma}\label{RD-Lemma1}
$\Q[\Q W]_0=\Q W_0$.
\end{lemma} 
Which gives (\ref{ConCLUSioN}) when combined with (\ref{ConCLUSioN1}), and the proof is complete.
\end{proof}

\medskip

Theorem \ref{RD-w} highlights the fact that the concept of $p$-amplitude has a ``nice" behavior with respect to the $\Gamma(\pi)$-convergence. More generally, let $\{\pi_\eps\}_\eps$ be a family of $L^p$-continuous maps $\pi_\eps$ from $W^{1,p}(\Sigma_\eps;\RR^m)$ to $W^{1,p}(\Sigma;\RR^m)$, where $\Sigma_\eps\subset\RR^N$ (resp. $\Sigma\subset\RR^k$ with $k\in\NN^*$) is a bounded open set,  let $\{W_\eps\}_{\eps}$ be an uniformly $p$-coercive family of measurable integrands $W_\eps:\MM^{m\times N}\to[0,+\infty]$ and, for each $\eps>0$, let $I_\eps, \Q I_\eps:W^{1,p}(\Sigma_\eps;\RR^m)\to[0,+\infty]$ be respectively defined by
\begin{itemize}
\item[\SMALL$\blacklozenge$] $\displaystyle I_\eps(\phi):=\int_{\Sigma_\eps}W_\eps(\nabla\phi(x))dx;$
\item[\SMALL$\blacklozenge$] $\displaystyle \Q I_\eps(\phi):=\int_{\Sigma_\eps}\Q W_\eps(\nabla\phi(x))dx$.
\end{itemize}
The following theorem says that the $\Gamma(\pi)$-limit is stable by substituting $I_\eps$ by $\Q I_\eps$ whenever every $W_\eps$ is $p$-ample.
\begin{theorem}
Assume that{\rm:}
\begin{itemize}
\item[\SMALL$\blacklozenge$] $\forall\eps>0$ $W_\eps$ is $p$-ample{\rm;}
\item[\SMALL$\blacklozenge$] $\exists I_0:W^{1,p}(\Sigma;\RR^m)\to[0,+\infty]$ $\Gamma(\pi)\hbox{-}\lim\limits_{\eps\to 0}\Q I_\eps=I_0$.
\end{itemize}
Then $\Gamma(\pi)\hbox{-}\lim\limits_{\eps\to 0} I_\eps=I_0$.
\end{theorem}
\begin{proof}
As every $W_\eps$ is $p$-ample, from Theorem \ref{TheoremA-B} we deduce that $\overline{I}_\eps=\Q I_\eps$ for all $\eps>0$. On the other hand, as every $\pi_\eps$ is $L^p$-continuous, it is easy to see that  $\Gamma(\pi)$-$\liminf_{\eps\to0} I_\eps=\Gamma(\pi)$-$\liminf_{\eps\to0} \overline{I}_\eps$ and $\Gamma(\pi)$-$\limsup_{\eps\to0} I_\eps=\Gamma(\pi)$-$\limsup_{\eps\to0} \overline{I}_\eps$, and the theorem follows.
\end{proof}

\subsection{$\Gamma(\pi)$-convergence of $I_\eps$: ``strong-Determinant Constraint"}
The following theorem gives an answer to Problem \ref{problem2} in the framework of hyperelasticity (it is consistent with the ``strong-Determinant Constraint", i.e., (\ref{s-DC})) in the same spirit as the works of Ball (see \cite{Ball77}), Acerbi-Buttazzo-Percivale (see \cite{ABP91}) and  Friesecke-James-M\"uller (see \cite{FJM02}). It is the result of several works on the subject: mainly, the attempt of Percivale in 1991 (see \cite{percivale91}), the rigorous answer to Problem \ref{problem2} by Le Dret and Raoult in the $p$-polynomial growth case (see \cite{ledret-raoult93,ledret-raoult95}) and especially the substantial contributions of Ben Belgacem (see \cite{benbelgacem96,benbelgacem97,benbelgacem00}).

\begin{theorem}[\cite{oah-jpm08b}]\label{RD-s}
Assume that{\rm:}
\begin{eqnarray}
 &&W\hbox{ is continuous\hbox{\rm ;}}\label{D0}\\
 &&W(F)=+\infty\iff\det F\leq 0;\label{D1}\\
&&\forall\delta>0\ \exists c_\delta>0\ \forall F\in\MM^{3\times 3}\big(\det F\geq\delta\then W(F)\leq c_\delta(1+|F|^p)\big).\label{D2}
\end{eqnarray}
Then
$$
\forall\psi\in W^{1,p}(\Sigma;\RR^3)\quad\Gamma(\pi)\hbox{-}\displaystyle\lim\limits_{\eps\to 0}I_\eps(\psi)=\int_\Sigma\Q W_0(\nabla\psi(x))dx.
$$
\end{theorem}
{\bf Proof.} It is easy to see that if $W$ satisfies (\ref{D0}), (\ref{D1}) and (\ref{D2}) then:
\begin{eqnarray}
&& W_0\hbox{ is continuous;}\label{P0}\\
&& W_0(\xi)=W(\xi\mid\xi_2)=+\infty\iff|\xi_1\land\xi_2|=0;\label{P01}\\
&& \forall\alpha>0\ \exists \beta_\alpha>0\ \forall \xi\in\MM^{3\times 2}\big(|\xi_1\land\xi_2|\geq\alpha\then W_0(\xi)\leq \beta_\alpha(1+|\xi|^p)\big).\label{P1}
\end{eqnarray}
In particular, $W_0$ satisfies (\ref{P}), i.e.,
$$
\exists\alpha,\beta>0\ \forall \xi=(\xi_1\mid \xi_2)\in\MM^{3\times 2}\ \big(|\xi_1\land \xi_2|\geq\alpha\then W_0(\xi)\leq\beta(1+|\xi|^p)\big),
$$
 since clearly (\ref{P1}) implies (\ref{P}). Let $\mathcal{I}, \overline{\mathcal{I}}, \overline{\mathcal{I}}_{\rm diff_*}:W^{1,p}(\Sigma;\RR^3)\to[0,+\infty]$ be respectively defined by:
\begin{itemize}
\item[\SMALL$\blacklozenge$] $\displaystyle\mathcal{I}(\psi):=\int_\Sigma W_0(\nabla\psi(x))dx$;
\item[\SMALL$\blacklozenge$] $\overline{\mathcal{I}}(\psi):=\inf\left\{\liminf\limits_{n\to+\infty}\mathcal{I}(\psi_n):\psi_n\stackrel{L^p}{\to}\psi\right\}$;
\item[\SMALL$\blacklozenge$] $\overline{\mathcal{I}}_{\rm diff_*}(\psi):=\inf\left\{\liminf\limits_{n\to+\infty}\mathcal{I}(\psi_n):C^1_*(\overline{\Sigma};\RR^3)\ni\psi_n\stackrel{L^p}{\to}\psi\right\}$,
\end{itemize}
where $C^1_*(\overline{\Sigma};\RR^3)$ is the set of $C^1$-immersions from $\overline{\Sigma}$ to $\RR^3$, i.e., $$C^1_*(\overline{\Sigma};\RR^3):=\Big\{\psi\in C^1(\overline{\Sigma};\RR^3):\forall x\in\overline{\Sigma}\ \partial_1\psi(x)\land\partial_2\psi(x)\not= 0\Big\}.$$ As $W_0$ satisfies (\ref{P}), by Corollary \ref{corollary1} we have
$$
\forall\psi\in W^{1,p}(\Sigma;\RR^3)\quad\overline{\mathcal{I}}(\psi)=\int_\Sigma\Q W_0(\nabla\psi(x))dx.
$$
On the other hand, we can establish the following two lemmas (whose the proofs are given in \S 3.6.2 and \S 3.6.3). 
\begin{lemma}[\cite{oah-jpm08b}]\label{RD-Lemma2}
$\overline{\mathcal{I}}\leq\Gamma\hbox{-}\liminf\limits_{\eps\to 0}\mathcal{I}_\eps$.
\end{lemma}

\begin{lemma}[\cite{oah-jpm08b}]\label{RD-Lemma3}
If {\rm(\ref{D0})}, {\rm(\ref{D1})} and {\rm(\ref{D2})} hold then $\Gamma\hbox{-}\limsup\limits_{\eps\to 0}\mathcal{I}_\eps\leq \overline{\mathcal{I}}_{\rm diff_*}$.
\end{lemma}
Hence, taking Lemma \ref{ABP} into account, it suffices to prove that 
\begin{equation}\label{I-0}
\overline{\mathcal{I}}_{\rm diff_*}\leq \overline{\mathcal{I}}.
\end{equation}
 Let us outline the proof of (\ref{I-0}) (a  more detailled proof is given in \S 3.6.5). Consider $\overline{\mathcal{I}}_{\rm aff_{\rm li}^{\rm reg}},\R\mathcal{I},\overline{\mathcal{R I}},\overline{\mathcal{R I}}_{\rm aff_{\rm li}^{\rm reg}}:W^{1,p}(\Sigma;\RR^3)\to[0,+\infty]$ respectively defined by:
\begin{itemize}
\item[\SMALL$\blacklozenge$] $\displaystyle \overline{\mathcal{I}}_{\rm aff_{\rm li}^{\rm reg}}(\psi):=\inf\left\{\liminf_{n\to+\infty}\mathcal{I}(\psi_n):\Aff_{\rm li}^{\rm reg}(\Sigma;\RR^3)\ni\psi_n\stackrel{L^p}{\to}\psi\right\}$;
\item[\SMALL$\blacklozenge$] $\displaystyle \R\mathcal{I}(\psi):=\int_\Sigma\R W_0(\nabla\psi(x))dx$;
\item[\SMALL$\blacklozenge$] $\displaystyle \overline{\mathcal{R I}}(\psi):=\inf\left\{\liminf_{n\to+\infty}\mathcal{R}\mathcal{I}(\psi_n):\psi_n\stackrel{L^p}{\to}\psi\right\}$;
\item[\SMALL$\blacklozenge$] $\displaystyle \overline{\mathcal{R I}}_{\rm aff_{\rm li}^{\rm reg}}(\psi):=\inf\left\{\liminf_{n\to+\infty}\mathcal{R}\mathcal{I}(\psi_n):\Aff_{\rm li}^{\rm reg}(\Sigma;\RR^3)\ni\psi_n\stackrel{L^p}{\to}\psi\right\},$
\end{itemize}
where $\Aff_{\rm li}^{\rm reg}(\Sigma;\RR^3)$ is a space of ``nice" locally injective continuous piecewise affine functions from $\Sigma$ to $\RR^3$ defined as follows.
\begin{definition}
 By a regular mesh in $\RR^2$ we mean a finite family  $\{V_i\}_{i\in I}$ of open disjoint triangles of $\RR^2$ such that for every  $i,j\in I$ with $i\not= j$, the intersection of $\overline{V_i}$ and $\overline{V_j}$ is either empty,  an edge of each or a vertices of each.  Given an open set $V\subset\RR^2$, we say that $\psi:V\to\RR^3$ is affine if it is the restriction to $V$ of an affine function from $\RR^2$ to $\RR^3$. The space of all continuous functions $\psi:\RR^2\to\RR^3$ for which there exists a regular mesh $\{V_i\}_{i\in I}$ in $\RR^2$ such that for every $i\in I$, $\psi\lfloor_{{V_i}}$ is affine and $\psi=0$ in $\RR^2\setminus\cup_{i\in I}\overline{V}_i$ is denoted by $\Aff^{\rm reg}_{\rm c}(\RR^2;\RR^3)$. We set:
\begin{itemize}
\item[] $\Aff^{\rm reg}(\Sigma;\RR^3):=\Big\{\psi\lfloor_{{\Sigma}}:\psi\in\Aff^{\rm reg}_{\rm c}(\RR^2;\RR^3)\Big\}$;
\item[] $\Aff^{\rm reg}_0(\Sigma;\RR^3):=\Big\{\psi\in\Aff^{\rm reg}(\Sigma;\RR^3):\psi=0\hbox{ on }\partial\Sigma\Big\}$.
\end{itemize}
We say that $\psi:\RR^2\to\RR^3$ is locally injective in $x\in\RR^2$ if there exists $\rho>0$ such that $\psi\lfloor_{{B_\rho(x)}}$ is injective, where $B_\rho(x)$ denotes the ball centered at  $x$ with radius $\rho$. Given $E\subset\RR^2$, when $\psi$ is locally injective in $x$ for all $x\in E$,  we say that $\psi$ is locally injective on $E$.  We set 
$$
\Aff_{\rm li}^{\rm reg}(\Sigma;\RR^3):=\Big\{\psi\lfloor_{{\Sigma}}:\Aff^{\rm reg}_{\rm c}(\RR^2;\RR^3)\ni\psi\hbox{ is locally injective on }\overline{\Sigma}\Big\}.
$$

\end{definition}

As $\overline{\mathcal{R I}}\leq\overline{\mathcal{I}}$, a way for proving (\ref{I-0}) is to establish the following  three inequalities:
\begin{eqnarray}
&&\overline{\mathcal{I}}_{\rm diff_*}\leq \overline{\mathcal{I}}_{\rm aff_{\rm li}^{\rm reg}};\label{I-1}\\
&&\overline{\mathcal{I}}_{\rm aff_{\rm li}^{\rm reg}}\leq \overline{\mathcal{R I}}_{\rm aff_{\rm li}^{\rm reg}};\label{I-2}\\
&&\overline{\mathcal{R I}}_{\rm aff_{\rm li}^{\rm reg}}\leq \overline{\mathcal{R I}}\label{I-3}.
\end{eqnarray}
The inequality (\ref{I-1}) follows by using the fact that $W_0$ satisfies (\ref{P0}) and (\ref{P1}) together with the following theorem due to Ben Belgacem and Bennequin (for a proof, see \cite[Lemme 8 p. 114]{benbelgacem96}; see also \cite[\S 4.2 p. 52]{oah-jpm09a} for a ``more analytic" proof).

\begin{theorem}[\cite{benbelgacem96}]\label{BBBapproxTheo}
For all $\psi\in\Aff_{\rm li}^{\rm reg}(\Sigma;\RR^3)$ there exists $\{\psi_n\}_{n\geq 1}\subset C^1_*(\overline{\Sigma};\RR^3)$ such that{\rm:}
\begin{eqnarray}
&&\psi_n\stackrel{W^{1,p}}{\to}\psi;\label{BBB1}\\
&&\exists\delta>0\ \forall x\in\overline{\Sigma}\ \forall n\geq 1\ |\partial_1\psi_n(x)\land\partial_2\psi_n(x)|\geq\delta.\label{BBB2}
\end{eqnarray}
\end{theorem}
The inequality (\ref{I-2}) is obtained by exploiting the Kohn-Strang representation of $\R W_0$. (Note that for establishing this inequality we need the assertion (\ref{P01}).) Finally, we establish  the inequality (\ref{I-3}) by combining the following two results: the first one is essentially due to Ben Belgacem (a proof is given in \S 3.6.4) and the second one to Gromov and {\`E}lia{\v{s}}berg (for a proof, see \cite[Theorem 1.3.4B]{gromov-eliashberg71}; see also \cite[\S 4.1 p. 44]{oah-jpm09a}). 

\begin{lemma}\label{BBLemma2}
If $W_0$ satisfies {\rm(\ref{P1})} then{\rm:}
\begin{itemize}
\item[\SMALL$\blacklozenge$] $\exists c>0\ \forall \xi\in\MM^{3\times 2}\ \R W_0(\xi)\leq c(1+|\xi|^p);$
\item[\SMALL$\blacklozenge$] $\mathcal{R} W_0$ is continuous.
\end{itemize}
\end{lemma}
\begin{theorem}[\cite{gromov-eliashberg71}]\label{TheoremGE}
$\Aff_{\rm li}^{\rm reg}(\Sigma;\RR^3)$ is strongly dense in $W^{1,p}(\Sigma;\RR^3)$. $\blacksquare$
\end{theorem}

\begin{question}
Try to simplify the proof of Theorem {\rm\ref{RD-s}} as follows{\rm:} first, approximate $W$ satisfying {\rm(\ref{D0})}, {\rm(\ref{D1})} and {\rm(\ref{D2})} or maybe weaker conditions compatible with the ``strong-Determinant Constraint", i.e., {\rm(\ref{s-DC})}, by a supremum of $p$-ample integrands $W_\delta$ satisfying {\rm(\ref{D})} with $\alpha,\beta>0$ which can depend on $\delta$, then, apply Theorem {\rm\ref{RD-w}} to each $W_\delta$, and finally, pass to the limit as $\delta$ goes to zero.
\end{question}

\subsection{Complementary proofs}

\subsubsection{Proof of Lemma \ref{RD-Lemma1}} It suffices to prove that \begin{equation}\label{EqRD-Lemma1}
\Z_\infty [\Z_\infty W]_0=\Z_\infty W_0.
\end{equation}
 Indeed, from Lemma \ref{R-Lemma3} we deduce that $\Z W_0$ is of $p$-polynomial growth, i.e., $\exists c>0\ \forall \xi\in\MM^{3\times 2}\ \Z W_0(\xi)\leq c(1+|\xi|^p)$, and so $\Z_\infty W_0$ is finite since $\Z_\infty W_0\leq \Z W_0$. Hence, $\Q W_0=\Z_\infty W_0$ by Theorem \ref{RTnon-finite-case}(i). On the other hand, $\Z_\infty W$ is of $p$-polynomial growth (see Lemma \ref{LeMMa-wDC}-bis), and so $\Q W=\Z_\infty W$ by Theorem \ref{RTnon-finite-case}(i). It follows that $[\Q W]_0=[\Z_\infty W]_0$ is finite and continuous, and so $\Q[\Q W]_0=\Q[\Z_\infty W]_0=\Z_\infty[\Z_\infty W]_0$ by Theorem \ref{DacorognaTheorem1}.
 
 Let us now prove (\ref{EqRD-Lemma1}). For any $\xi\in\MM^{3\times 2}$, $\Z_\infty[\Z_\infty W]_0(\xi)\leq[\Z_\infty W]_0(\xi)\leq\Z_\infty W(\xi\mid\zeta)\leq W(\xi\mid\zeta)$ for all $\zeta\in\RR^3$, and so $\Z_\infty[\Z_\infty W]_0(\xi)\leq W_0(\xi)$ for all $\xi\in\MM^{3\times 2}$, i.e., $\Z_\infty[\Z_\infty W]_0\leq\Z_\infty W_0$. It follows that $\Z_\infty[\Z_\infty W]_0\leq\Z_\infty W_0$. It then remains to prove that 
 \begin{equation}\label{RevERseInequALITY}
 \Z_\infty[\Z_\infty W]_0\ge \Z_\infty W_0.
 \end{equation}
Given $\delta>0$ et $\xi\in\MM^{3\times 2}$, there exist $\zeta\in\RR^3$ and $\varphi\in W^{1,\infty}_0(Y;\RR^3)$ (with $Y:=]0,1[^3$) such that 
\[
\left[\Z_\infty W\right]_0(\xi)+\delta\ge \int_YW\big(\xi+\nabla\varphi_{x_3}(x)\mid\zeta+\partial_3\varphi(x,x_3)\big)dxdx_3
\]
with $\varphi_{x_3}\in \Aff_0(]0,1[^2;\RR^3)$ defined by $\varphi_{x_3}(x):=\varphi(x,x_3)$. But
\begin{eqnarray*}
\int_YW\big(\xi+\nabla\varphi_{x_3}(x)\mid\zeta+\partial_3\varphi(x,x_3)\big)dxdx_3&\ge& \int_{0}^{1}\int_{]0,1[^2} W_0(\xi+\nabla\varphi_{x_3}(x))dxdx_3\\
&\ge&\int_{0}^{1}\Z_\infty W_0(\xi)dx_3=\Z_\infty  W_0(\xi),
\end{eqnarray*}
hence $[\Z_\infty W]_0(\xi)+\delta\ge\Z_\infty W_0(\xi)$, and consequently  $[\Z_\infty W]_0(\xi)\ge \Z_\infty W_0(\xi)$ by letting $\delta\to 0$. Thus $[\Z_\infty W]_0\ge \Z_\infty W_0$, and (\ref{RevERseInequALITY}) follows. $\blacksquare$


\subsubsection{Proof of Lemma \ref{RD-Lemma2}} Let $\psi\in W^{1,p}(\Sigma;\RR^3)$ and let $\{\psi_\eps\}_\eps\subset W^{1,p}(\Sigma;\RR^3)$ be such that $\psi_\eps\to \psi$ in $L^{p}(\Sigma;\RR^3)$. We have to prove that
\begin{equation}\label{main_ineq1}
\liminf_{\eps\to 0}\mathcal{I}_\eps(\psi_\eps)\geq \overline{\mathcal{I}}(v).
\end{equation}
Without loss of generality we can assume that $\sup_{\eps>0}\mathcal{I}_\eps(\psi_\eps)<+\infty$. To every  $\eps>0$ there corresponds $\phi_\eps\in\pi_{\eps}^{-1}(\psi_\eps)$ such that
\begin{equation}\label{main_ineq2}
\mathcal{I}_\eps(\psi_\eps)\geq I_\eps(\phi_\eps)-\eps.
\end{equation}
Defining $\hat \phi_\eps:\Sigma_1\to\RR^3$ by $\hat \phi_\eps(x,x_3):=\phi_\eps(x,\eps x_3)$ (with $\Sigma_1=\Sigma\times]-{1\over 2},{1\over 2}[$) we have 
\begin{equation}\label{main_ineq3}
I_\eps(\phi_\eps)=\int_{\Sigma_1}W\left(\partial_{1} \hat \phi_\eps(x,x_3)\mid\partial_{2} \hat \phi_\eps(x,x_3)\mid{1\over\eps}\partial_3 \hat \phi_\eps(x,x_3)\right)dxdx_3.
\end{equation}
Using the coercivity of  $W$, we deduce that $\|{\partial_3 \hat \phi_\eps}\|_{L^p(\Sigma_1;\RR^3)}\leq c\eps^{p}$ for all $\eps>0$ and some $c>0$, and so
$
\|\hat \phi_\eps-\psi_\eps\|_{L^p(\Sigma_1;\RR^3)}\leq c^\prime\eps^p
$
by the  Poincar\'e-Wirtinger inequality, where $c^\prime>0$ is a constant which does not depend on $\eps$. It follows that $\hat \phi_\eps\to \psi$ in $L^p(\Sigma_1;\RR^3)$. For $x_3\in]-{1\over 2},{1\over 2}[$, let $\varphi_\eps^{x_3}\in W^{1,p}(\Sigma;\RR^3)$ be defined by $\varphi_\eps^{x_3}(x):=\hat \phi_\eps(x,x_3)$. Then (up to a subsequence) $\varphi_\eps^{x_3}\to \psi$ in $L^p(\Sigma;\RR^3)$ for a.e. $x_3\in ]-{1\over 2},{1\over 2}[$. Taking (\ref{main_ineq2}) and (\ref{main_ineq3}) into account and using the Fatou lemma, we obtain 
$$
\liminf_{\eps\to 0}{\mathcal{I}}_\eps(\psi_\eps)\geq\int_{-{1\over 2}}^{1\over 2}\left(\liminf_{\eps\to 0}\int_\Sigma W_0\big(\nabla \varphi_\eps^{x_3}(x)\big)dx\right)dx_3,
$$
and (\ref{main_ineq1}) follows.  $\blacksquare$


\subsubsection{Proof of Lemma \ref{RD-Lemma3}} Given $\psi\in C^1_*(\overline{\Sigma};\RR^3)$ and $j\geq 1$, define $\Lambda_\psi^j:\overline{\Sigma}\dto\RR^3$ by :
$$
\Lambda^j_\psi(x):=\left\{\zeta\in\RR^3:\det(\nabla\psi(x)\mid\zeta)\geq{1\over j}\right\}.
$$
It is easy to see that:
 \begin{eqnarray}
 &&\Lambda^j_\psi\hbox{ is a nonempty convex closed-valued semicontinuous\footnote{} multifunction;}\label{Det>0Prop0}\\
&&\Lambda^1_\psi(x)\subset\Lambda^2_\psi(x)\subset\Lambda^3_\psi(x)\subset\cdots\subset\cupp\limits_{j\geq 1}\Lambda^j_\psi(x)=\big\{\zeta\in\RR^3:\det(\nabla\psi(x)\mid\zeta)>0\big\}.\label{Det>0Prop1}
\end{eqnarray}
\footnotetext[2]{A multifunction $\Lambda:\overline{\Sigma}\to\RR^3$ is said to be lower semicontinuous if for every closed subset $X$ of $\RR^3$, every $x\in\overline{\Sigma}$ and every $\{x_n\}_{n\geq 1}\subset\overline{\Sigma}$ such that $|x_n-x|\to 0$ as $n\to+\infty$ and $\Lambda(x_n)\subset X$ for all $n\geq 1$, we have $\Lambda(x)\subset X$ (see \cite{aubin-frankowska90} for more details).}
In the sequel, given $\Lambda:\overline{\Sigma}\dto\RR^3$ we set
$$
C(\overline{\Sigma};\Lambda):=\Big\{\phi\in C\big(\overline{\Sigma};\RR^3\big):\phi(x)\in\Lambda(x)\hbox{ for all }x\in\overline{\Sigma}\Big\},
$$ 
where $C(\overline{\Sigma};\RR^3)$ denotes the space of all continuous functions from $\overline{\Sigma}$ to $\RR^3$. 
\begin{lemma}\label{LemmeDeT>0}
Let $\psi\in C^1_*(\overline{\Sigma};\RR^3)$ and $j\geq 1$. If $W$ is continuous and satisfies {\rm(\ref{D2})} then
$$
\inf_{\varphi\in C(\overline{\Sigma};\Lambda_\psi^j)}\int_\Sigma W(\nabla\psi(x)\mid\varphi(x))dx=\int_\Sigma \inf_{\zeta\in\Lambda_\psi^j(x)}W(\nabla\psi(x)\mid\zeta)dx.
$$
\end{lemma}
To prove Lemma \ref{LemmeDeT>0} we need the following interchange theorem of infimum and integral (for a proof, see \cite[Corollary 5.4]{oah-jpm03}; see also \cite[\S 5.2 p. 60]{oah-jpm09a}). 
\begin{theorem}[\cite{oah-jpm03}]\label{AHM} Let $\Lambda:\overline{\Sigma}\dto\RR^3$ and let $f:\overline{\Sigma}\times\RR^3\to[0,+\infty]$. Assume that{\rm:} 
\begin{itemize}
\item[(H$_1$)]  $f$ is a Carath\'eodory integrand{\rm;}
\item[(H$_2$)] $\Lambda$ is a nonempty convex closed-valued lower semicontinuous multifunction{\rm;}
\item[(H$_3$)] 
$
\int_\Sigma \max_{\alpha\in[0,1]}f\big(x,\alpha\varphi(x)+(1-\alpha)\hat\varphi(x)\big)dx<+\infty.
$
for all $\varphi,\hat\varphi\in C(\overline{\Sigma};\Lambda)$.
\end{itemize}
Then,
$$
\inf_{\varphi\in C(\overline{\Sigma};\Lambda)}\int_{\Sigma}f\big(x,\varphi(x)\big)dx=\int_\Sigma\inf_{\zeta\in\Lambda(x)}f(x,\zeta)dx.
$$
\end{theorem}

\begin{proof}[\bf Proof of Lemma \ref{LemmeDeT>0}]
Since $W$ is continuous, (H$_1$) is satisfied with $f(x,\zeta)=W(\nabla\psi(x)\mid\zeta)$. Furthermore, taking (\ref{Det>0Prop0}) into account, we see that (H$_2$) holds with $\Lambda=\Lambda^j_\psi$. On the other hand, given $\varphi,\hat\varphi\in C(\overline{\Sigma};\Lambda^j_\psi)$, it is clear that $\Det(\nabla\psi(x)\mid\alpha\varphi(x)+(1-\alpha)\hat\varphi(x))\geq{1\over j}$ for all $\alpha\in[0,1]$ and all $x\in\overline{\Sigma}$. Using (\ref{D2}) we can assert that there exists $c>0$ (depending only on $j,\psi,\varphi$ and $\hat\varphi$) such that $W(\nabla\psi(x)\mid\alpha\varphi(x)+(1-\alpha)\hat\varphi(x))\leq c$  for all $\alpha\in[0,1]$ and all $x\in\overline{\Sigma}$. Thus (H$_3$) is satisfied with $f(x,\zeta)=W(\nabla\psi(x)\mid\zeta)$ and $\Lambda=\Lambda^j_\psi$, and Lemma \ref{LemmeDeT>0} follows from Theorem \ref{AHM}. \end{proof}

\medskip

The following lemma gives a ``non-integral" representation for $\mathcal{I}$ on $C^1_*(\overline{\Sigma};\RR^3)$.
\begin{lemma}\label{IntermediateTheorDetNot}
If $W$ satisfies {\rm(\ref{D0})} and {\rm(\ref{D2})} and if $\psi\in C^1_*(\overline{\Sigma};\RR^3)$ then
$$
\mathcal{I}(\psi)=\inf_{j\geq 1}\inf_{\varphi\in C(\overline{\Sigma};\hat\Lambda^j_\psi)}\int_\Sigma W(\nabla\psi(x)\mid\varphi(x))dx.
$$
\end{lemma}
\begin{proof}[\bf Proof of Lemma \ref{IntermediateTheorDetNot}] It suffices to prove that
\begin{eqnarray}\label{ineQuality}
\mathcal{I}(\psi)\geq\inf_{j\geq 1}\inf_{\varphi\in C(\overline{\Sigma};\hat\Lambda^j_\psi)}\int_\Sigma W(\nabla\psi(x)\mid\varphi(x))dx.
\end{eqnarray}
Using Lemma \ref{LemmeDeT>0}, we obtain
\begin{equation}\label{inequality1}
\inf_{j\geq 1}\inf_{\varphi\in C(\overline{\Sigma};\hat\Lambda^j_\psi)}\int_\Sigma W(\nabla\psi(x)\mid\varphi(x))dx\leq\inf_{j\geq 1}\int_{\Sigma}\inf_{\zeta\in\Lambda^j_\psi(x)}W\big(\nabla \psi(x)\mid\zeta\big)dx.
\end{equation}
Consider the continuous function $\Phi:\overline{\Sigma}\to\RR^3$ defined by 
\begin{equation}\label{ExampleOfContinuousSelection}
\Phi(x):={{\partial_1 \psi(x)\land\partial_2 \psi(x)}\over \vert\partial_1\psi(x)\land\partial_2\psi(x)\vert^2}.
\end{equation}
 Then, $\det(\nabla \psi(x)\mid\Phi(x))=1$ for all $x\in\overline{\Sigma}$. Using {\rm (\ref{D2})} we deduce that there exists $c>0$ depending only on $p$ such that
$$
\int_\Sigma\inf_{\zeta\in\Lambda^{1}_\psi(x)}W(\nabla \psi(x)\mid\zeta)dx\leq c\big(|\Sigma|+\|\nabla \psi\|^p_{L^p(\Sigma;\MM^{3\times 2})}+\|\Phi\|^p_{L^p(\Sigma;\RR^3)}\big).
$$ 
It follows that $\inf_{\zeta\in\Lambda^{1}_\psi(\cdot)}W(\nabla \psi(\cdot)\mid\zeta)\in L^1(\Sigma)$. From (\ref{Det>0Prop0}) and (\ref{Det>0Prop1}), we see that $\{\inf_{\zeta\in\Lambda^j_\psi(\cdot)}W(\nabla \psi(\cdot)\mid\zeta)\}_{j\geq 1}$ is non-increasing and  
 \begin{equation}\label{equality}
\inf_{j\geq 1}\inf_{\zeta\in\Lambda^j_\psi(x)}W\big(\nabla \psi(x)\mid\zeta\big)=W_0\big(\nabla \psi(x)\big)
\end{equation}
for all $x\in \overline{\Sigma}$,
 and (\ref{ineQuality}) follows from (\ref{inequality1}) and ({\ref{equality}}) by using the Lebesgue monotone convergence theorem.
\end{proof}

\medskip

We can now prove Lemma \ref{RD-Lemma3}. As $\Gamma\hbox{-}\limsup_{\eps\to 0}\mathcal{I}_\eps$ is lower semicontinuous with respect to the strong topology of $L^{p}(\Sigma;\RR^3)$, it is sufficient to prove that 
\begin{equation}\label{limsupequality}
\limsup_{\eps\to 0}\mathcal{I}_\eps(\psi)\leq\mathcal{I}(\psi)
\end{equation}  
for all $\psi\in C^1_*(\overline{\Sigma};\RR^3)$. Given $\psi\in C^1_*(\overline{\Sigma};\RR^3)$, fix any   $j\geq 1$ and any  $n\geq 1$. Using Lemma \ref{IntermediateTheorDetNot} we obtain the existence of $\varphi\in C(\overline{\Sigma};\hat\Gamma^j_\psi)$ such that 
\begin{equation}\label{mediainequality}
\int_\Sigma W(\nabla \psi(x)\mid\varphi(x))dx\leq\mathcal{I}(\psi)+{1\over n}.
\end{equation}
Let $\{\varphi_k\}_{k\geq 1}\subset C^\infty(\overline{\Sigma};\RR^3)$ be such that 
\begin{equation}\label{uniformityconvergence}
\varphi_k\to\varphi\hbox{ uniformly as }k\to+\infty.
\end{equation}
We claim that:
\begin{eqnarray}\label{CoNdItioN(c1)}
&&\det(\nabla \psi(x)\mid\varphi_k(x))\geq{1\over 2j}\hbox{ for all }x\in \overline{\Sigma},\hbox{ all }k\geq k_\psi\hbox{ and some }k_\psi\geq 1;\\
&&\lim\limits_{k\to+\infty}\int_\Sigma W(\nabla\psi(x)\mid\varphi_k(x))dx=\int_\Sigma W(\nabla\psi(x)\mid\varphi(x))dx.\label{CoNdItioN(c2)}
\end{eqnarray}
Indeed, setting $\mu_\psi:=\sup_{x\in V}|\partial_1\psi(x)\land\partial_2\psi(x)|=\max_{i\in I}|\xi_{i,1}\land\xi_{i,2}|$ ($\mu_\psi>0$) and using (\ref{uniformityconvergence}), we deduce that there exists $k_\psi\geq 1$ such that
\begin{equation}\label{supequality}
\sup_{x\in\overline{\Sigma}}|\varphi_k(x)-\varphi(x)|<{1\over 2j\mu_\psi}
\end{equation}
for all $k\geq k_\psi$. Let $x\in \overline{\Sigma}$ and let  $k\geq k_\psi$. As $\varphi\in C(\overline{\Sigma};\hat\Gamma^j_\psi)$ we have
\begin{equation}\label{supequality1}
\Det(\nabla\psi(x)\mid\varphi_k(x))\geq{1\over j}-\Det(\nabla\psi(x)\mid\varphi_k(x)-\varphi(x)).
\end{equation}
Noticing that $\Det(\nabla\psi(x)\mid\varphi_k(x)-\varphi(x))\leq|\partial_1\psi(x)\land\partial_2\psi(x)||\varphi_k(x)-\varphi(x)|$, from (\ref{supequality}) and (\ref{supequality1}) we deduce that
$
\Det(\nabla\psi(x)\mid\varphi_k(x))\geq{1\over 2j}
$
and (\ref{CoNdItioN(c1)}) is proved. Combining (\ref{CoNdItioN(c1)}) and (\ref{D2}) we see that
$
\sup_{k\geq k_\psi}W(\nabla \psi(\cdot)\mid\varphi_k(\cdot))\in L^1(\Sigma).
$ 
As $W$ is continuous we have 
$
\lim_{k\to+\infty}W(\nabla \psi(x)\mid\varphi_k(x))=W(\nabla \psi(x)\mid\varphi(x))
$
for all $x\in \overline{\Sigma}$, and  (\ref{CoNdItioN(c2)}) follows by the Lebesgue dominated convergence theorem.

Fix any $k\geq k_\psi$ et define the continuous function  $\theta:]-{1\over 2},{1\over 2}[\to\RR$ by 
$
\theta(x_3):=\inf_{x\in \overline{\Sigma}}\Det(\nabla\psi(x)+x_3\nabla\varphi_k(x)\mid\varphi_k(x)).
$
By (\ref{CoNdItioN(c1)}) we have $\theta(0)\geq{1\over 2j}$ and so there exists $\eta_\psi\in]0,{1\over 2}[$ such that $\theta(x_3)\geq{1\over 4j}$ for all $x_3\in]-\eta_\psi,\eta_\psi[$. Let $\phi_k:\Sigma_1\to\RR$ be given by
$
\phi_k(x,x_3):=\psi(x)+x_3\varphi_k(x).
$
From the above it follows that
\begin{eqnarray}\label{CoNdItioN(c3)}
&& \Det(\nabla \phi_k(x,\eps x_3))\geq{1\over 4j}\hbox{ for all }\eps\in]0,\eta_\psi[\hbox{ and all }(x,x_3)\in \overline{\Sigma}\times]-{1\over 2},{1\over 2}[.
\end{eqnarray}
As in the proof of (\ref{CoNdItioN(c2)}), combining (\ref{CoNdItioN(c3)}) et (\ref{D2}) and using the continuity of $W$, we obtain
\begin{equation}\label{finalequality}
\lim_{\eps\to 0}I_\eps(\phi_k)=\lim_{\eps\to 0}\int_{\Sigma_1} W(\nabla \phi_k(x,\eps x_3))dxdx_3=\int_\Sigma W(\nabla \psi(x)\mid\varphi_k(x))dx.
\end{equation}

Since $\pi_\eps(\phi_k)=\psi$, $\mathcal{I}_\eps(\psi)\leq I_\eps(\phi_k)$ for all $\eps>0$ and all $k\geq k_\psi$. Using (\ref{finalequality}), (\ref{CoNdItioN(c2)}) and (\ref{mediainequality}), we deduce that
$
\limsup_{\eps\to 0}\mathcal{I}_\eps(\psi)\leq\mathcal{I}(\psi)+{1\over n},
$
and (\ref{limsupequality}) follows by letting $n\to+\infty$. $\blacksquare$


\subsubsection{Proof of Lemma \ref{BBLemma2}} In what follows $N\leq m$ and given $F\in\MM^{m\times N}$,  $0\leq v_1(F)\leq\cdots\leq v_N(F)$ denote the singular values of $F$. Set 
$$
v(F):=\prod_{i=1}^Nv_i(F).
$$
When $N=2$ and $m=3$, it easy to check that $v(F)=|F_1\land F_2|$ for all $F=(F_1\mid F_2)\in\MM^{3\times 2}$. Recalling that any finite rank-one convex function is continuous, we see that Lemma \ref{BBLemma2} is a direct consequence of the following theorem.

\begin{theorem}\label{SingularValues}
Assume that 
\begin{equation}\label{RWaCCbyBBCondition}
\exists\alpha,\beta>0\ \forall F\in\MM^{m\times N}\ \big(v(F)\geq \alpha\then W(F)\leq\beta(1+|F|^p\big).
\end{equation}
Then $\R W$ is of $p$-polynomial growth, i.e., 
$$
\exists c>0\ \ \forall F\in\MM^{m\times N}\ \R W(F)\leq c(1+|F|^p).
$$
\end{theorem}
\begin{proof}[\bf Proof of theorem \ref{SingularValues}]
Without loss of generality we can assume that $\alpha\geq 1$. It is clear that $\R W(F)\leq\beta(1+|F|^p)$ for all $F\in\MM^{m\times N}$ such that $v(F)\geq \alpha$. Consider then $F\in\MM^{m\times N}$ such that $v(F)<\alpha$.  Let $P\in\OO(m)$ be such that 
$$
F=PJU,
$$
 where $U:=\sqrt{F^{\rm T}F}$ and $J=(J_{ij})\in\MM^{m\times N}$ with $J_{ij}=0$ if $i\not=j$ and $J_{ii}=1$, and let $Q\in\SO(N)$ be such that 
 $$
 U=Q^{\rm T}\diag(v_1(F),\cdots,v_N(F)) Q.
 $$
   Then:
\begin{itemize}
\item[\SMALL$\blacklozenge$] $F=PJQ^{\rm T}\diag(v_1(F),\cdots,v_N(F)) Q$;
\item[\SMALL$\blacklozenge$] $|F|^2=\sum_{i=1}^Nv_i^2(F)$.
\end{itemize}
Since $v(F)<\alpha$, there exists $1\leq i_1\leq\cdots\leq i_k\leq N$ with $k\in\{1,\cdots,N\}$ such that $v_{i_1}(F)<\alpha,\cdots,v_{i_k}(F)<\alpha$ (and $v_i(F)\geq\alpha$ for all $i\not\in\{i_1,\cdots,i_k\}$). For every  $j\in\{1,\cdots,k\}$, let $t_j\in]0,1[$ be such that 
$
v_{i_j}(F)=(1-t_j)(-\alpha)+t_j\alpha.
$
Then
\begin{eqnarray*}
\diag(v_1(F),\cdots,v_{i_1}(F),\cdots,v_N(F))&=&(1-t_1)\diag(v_1(F),\cdots,-\alpha,\cdots,v_N(F))\\ &&+\ t_1\diag(v_1(F),\cdots,\alpha,\cdots,v_N(F)),
\end{eqnarray*}
and so $F=(1-t_1)F_{1}^-+t_1F^+_{1}$ with:
\begin{itemize}
\item[\SMALL$\blacklozenge$] $F^-_1:=PJQ^{\rm T}\diag(v_1(F),\cdots,-\alpha,\cdots,v_N(F)) Q$;
\item[\SMALL$\blacklozenge$] $F^+_1:=PJQ^{\rm T}\diag(v_1(F),\cdots,\alpha,\cdots,v_N(F)) Q$;
\item[\SMALL$\blacklozenge$] $\rank(F^-_1-F^+_1)=1$.
\end{itemize}
Moreover, we have
\begin{eqnarray*}
&&\diag(v_1(F),\cdots,-\alpha,\cdots,v_{i_2}(F),\cdots,v_N(F))=
(1-t_2)\diag(v_1(F),\cdots,-\alpha,\\ &&\cdots,-\alpha,\cdots,v_N(F))
+t_2\diag(v_1(F),\cdots,-\alpha,\cdots,\alpha,\cdots,v_N(F)),
\end{eqnarray*}
hence $F^-_1=(1-t_2)F^{-,-}_2+t_2F^{-,+}_2$ with:
\begin{itemize}
\item[\SMALL$\blacklozenge$] $F^{-,-}_2:=PJQ^{\rm T}\diag(v_1(F),\cdots,-\alpha,\cdots,-\alpha,\cdots,v_N(F)) Q$;
\item[\SMALL$\blacklozenge$] $F^{-,+}_2:=PJQ^{\rm T}\diag(v_1(F),\cdots,-\alpha,\cdots,\alpha,\cdots,v_N(F)) Q$;
\item[\SMALL$\blacklozenge$] $\rank(F^{-,-}_2-F^{-,+}_2)=1$.
\end{itemize}
In the same manner, we obtain $F^+_1=(1-t_2)F^{+,-}_2+t_2F^{+,+}_2$ with:
\begin{itemize}
\item[\SMALL$\blacklozenge$] $F^{+,-}_2:=PJQ^{\rm T}\diag(v_1(F),\cdots,\alpha,\cdots,-\alpha,\cdots,v_N(F)) Q$;
\item[\SMALL$\blacklozenge$] $F^{+,+}_2:=PJQ^{\rm T}\diag(v_1(F),\cdots,\alpha,\cdots,\alpha,\cdots,v_N(F)) Q$;
\item[\SMALL$\blacklozenge$] $\rank(F^{+,-}_2-F^{+,+}_2)=1$.
\end{itemize}
We continue in this fashion obtaining a finite sequence  $\{F_j^{\sigma}\}_{j\in\{1,\cdots,k\}}^{\sigma\in\mathfrak{S}_j}\subset\MM^{m\times N}$, where $\mathfrak{S}_j$ denotes the set of all  maps  $\sigma:\{1,\cdots,j\}\to\{-,+\}$, with the following properties:
\begin{itemize}
\item[\SMALL$\blacklozenge$]  $F_j^\sigma=PJQ^{\rm T}\diag(v_1(F),\cdots,\sigma(1)\alpha,\cdots,\sigma(j)\alpha,\cdots,v_N(F))$; 
\item[\SMALL$\blacklozenge$] if $\sigma(j)\not=\sigma^\prime(j)$ and $\sigma(l)=\sigma^\prime(l)$ for all $l\in\{1,\cdots,j-1\}$ then $\rank(F_j^\sigma-F_j^{\sigma^\prime})=1$;
\item[\SMALL$\blacklozenge$] if $\sigma(1)\not=\sigma^\prime(1)$ then $F=(1-t_1)F_1^\sigma+t_1F^{\sigma^\prime}_1$;
\item[\SMALL$\blacklozenge$] if $\sigma^\prime(j+1)\not=\sigma^{\prime\prime}(j+1)$ and $\sigma^\prime(l)=\sigma^{\prime\prime}(l)=\sigma(l)$ for all $l\in\{1,\cdots,j\}$, then $F_j^\sigma=(1-t_{j+1})F_{j+1}^{\sigma^\prime}+t_{j+1}F_{j+1}^{\sigma^{\prime\prime}}$.\end{itemize}
It follows that:
\begin{itemize}
\item[\SMALL$\blacklozenge$] if $\sigma(1)\not=\sigma^\prime(1)$ then $\R W(F)\leq\R W(F^\sigma_1)+\R W(F^{\sigma^\prime}_1)$; \item[\SMALL$\blacklozenge$] if $\sigma^\prime(j+1)\not=\sigma^{\prime\prime}(j+1)$ and $\sigma^\prime(l)=\sigma^{\prime\prime}(l)=\sigma(l)$ for all $l\in\{1,\cdots,j\}$, then $\R W(F^\sigma_j)\leq\R W(F_{j+1}^{\sigma^\prime})+\R W(F_{j+1}^{\sigma^{\prime\prime}})$. \end{itemize}
Hence
$$
\R W(F)\leq\sum_{\sigma\in\mathfrak{S}_k}\R W(F^\sigma_k).
$$
Moreover, we have
\begin{eqnarray*}
v(F^\sigma_k)&=&\big|\det(\diag(v_1(F),\cdots,\sigma(1)\alpha,\cdots,\sigma(k)\alpha,\cdots,v_N(F)))\big|\\
&=&\alpha^k\prod_{i\not\in\{i_1,\cdots i_k\}}v_i(F),
\end{eqnarray*}
and so $v(F^\sigma_k)\geq \alpha^N\geq\alpha$ for all $\sigma\in\mathfrak{S}_k$. Using (\ref{RWaCCbyBBCondition}) we deduce that 
$$
\R W(F)\leq\sum_{\sigma\in\mathfrak{S}_k}\beta(1+|F^\sigma_k|^p).
$$
But
\begin{eqnarray*}
|F^\sigma_k|^2&=&\big|\diag(v_1(F),\cdots,\sigma(1)\alpha,\cdots,\sigma(k)\alpha,\cdots,v_N(F))\big|^2\\
&=&k\alpha^2+\sum_{i\not\in\{i_1,\cdots,i_k\}}v_i^2(F)\leq N\alpha^2+|F|^2,
\end{eqnarray*} 
hence
$$
\R W(F)\leq \sum_{\sigma\in\mathfrak{S}_k}\beta\big(1+2^{p\over 2}(N^{p\over 2}\alpha^p+|F|^p)\big)\leq c(1+|F|^p)
$$
with $c=2^N\beta(1+2^{p\over 2}N^{p\over 2}\alpha^p)$, which is the desired conclusion.
\end{proof}
\begin{remark}\label{RemarkSingularValues}
When $m=N$, it is easy to check that $v(F)=|\det F|$ for all $F\in\MM^{N\times N}$. Consequently, if $W$ satisifies (\ref{D}), i.e., 
$$
\exists\alpha,\beta>0\ \forall F\in\MM^{N\times N}\ \big(|\det F|\geq\alpha\then W(F)\leq\beta(1+|F|^p)\big),
$$
then $\R W$ is of $p$-polynomial growth, i.e., $\exists c>0\ \forall F\in\MM^{N\times N}\ \Z W(F)\leq c(1+|F|^p)$.
\end{remark}


\subsubsection{Proof of the inequality (\ref{I-0})} It suffices to prove the inequalities (\ref{I-1}), (\ref{I-2}) and (\ref{I-3}). On the other hand, it clear that:
\begin{itemize}
\item[\SMALL$\blacklozenge$] if
\begin{equation}\label{D0bis}
\overline{\mathcal{I}}_{\rm diff_*}(\psi)\leq\int_{\Sigma}W_0(\nabla \psi(x))dx\hbox{ for all }\psi\in \Aff_{\rm li}^{\rm reg}(\Sigma;\RR^3)
\end{equation}
then (\ref{D0}) holds;
\item[\SMALL$\blacklozenge$] if
\begin{equation}\label{D1bis}
{\mathcal{I}}_{\Aff_{\rm li}^{\rm reg}}(\psi)\leq\int_\Sigma\mathcal{R}W_0(\nabla\psi(x))dx\hbox{ for all }\psi\in\Aff_{\rm li}^{\rm reg}(\Sigma;\RR^3)
\end{equation}
then (\ref{D1}) holds;
\item[\SMALL$\blacklozenge$]if
\begin{equation}\label{D2bis}
\overline{\mathcal{RI}}_{\Aff_{\rm li}^{\rm reg}}(\psi)\leq\int_{\Sigma}\mathcal{R}W_0(\nabla \psi(x))dx\hbox{ for all }\psi\in W^{1,p}(\Sigma;\RR^3)
\end{equation}
then (\ref{D2}) holds.
\end{itemize}
Hence, we only need to show (\ref{D0bis}), (\ref{D1bis}) and (\ref{D2bis}).

\begin{proof}[\bf Proof of (\ref{D0bis})]
Let $\psi\in \Aff_{\rm li}^{\rm reg}(\Sigma;\RR^3)$. By Theorem \ref{BBBapproxTheo} there exists $\{\psi_n\}_{n\geq 1}\subset C^1_*(\overline{\Sigma};\RR^3)$ such that  (\ref{BBB1}) and (\ref{BBB2}) holds and $\nabla \psi_n(x)\to\nabla \psi(x)$ a.e. in $\Sigma$. As $W_0$ satisfies (\ref{P01}), i.e., $W_0$ is continuous, we have 
$$
\lim_{n\to +\infty}W_0\big(\nabla \psi_n(x)\big)=W_0\big(\nabla \psi(x)\big)\;\hbox{ a.e. in }\Sigma.
$$
Using (\ref{P1}) together with (\ref{BBB2}) we deduce that there exists $c>0$ such that for every $n\geq 1$ and every measurable set  $A\subset\Sigma$,
$$
\int_A W_0\big(\nabla \psi_n(x)\big)dx\leq c\Big(|A|+\int_A|\nabla \psi_n(x)-\nabla \psi(x)|^pdx+\int_A|\nabla \psi(x)|^pdx\Big).
$$
But $\nabla \psi_n\to\nabla \psi$ in $L^p(\Sigma;\MM^{3\times 2})$ by (\ref{BBB1}), hence $\{W_0(\nabla \psi_n(\cdot))\}_{n\geq 1}$ is absolutely uniformly integrable. Using the Vitali theorem, we obtain 
$$
\lim_{n\to+\infty}\int_{\Sigma}W_0(\nabla \psi_n(x))dx=\int_{\Sigma}W_0(\nabla \psi(x))dx,
$$ 
and (\ref{D0bis}) follows.
\end{proof}

\begin{proof}[\bf Proof of (\ref{D2bis})]
Let $\psi\in W^{1,p}(\Sigma;\RR^3)$. By Theorem \ref{TheoremGE} there exists $\{\psi_n\}_{n\geq 1}\subset\Aff_{\rm li}^{\rm reg}(\Sigma;\RR^3)$ such that $\nabla \psi_n\to\nabla \psi$ in $L^p(\Sigma;\RR^3)$ and $\nabla \psi_n(x)\to\nabla \psi(x)$ a.e. in $\Sigma$. Taking Lemma \ref{BBLemma2} into account, from the Vitali theorem, we deduce that 
$$
\lim_{n\to+\infty}\int_{\Sigma}\mathcal{R}W_0(\nabla \psi_n(x))dx=\int_{\Sigma}\mathcal{R}W_0(\nabla \psi(x))dx,
$$
and (\ref{D2bis}) follows.
\end{proof}

\begin{proof}[\bf Proof of (\ref{D1bis})]
We begin with some preliminaries: mainly, we state five lemmas. The proof of the first lemma (which is due to Kohn and Strang) will be omitted while the four others lemma will be proved below. Define the sequence $\{\R_i W_0\}_{i\geq 0}$ by $\R_0 W_0=W_0$ and for every $i\geq 1$ and every $\xi\in\MM^{3\times 2}$,
$$
\R_{i+1}W_0(\xi):=\infff\limits_{t\in[0,1]}\big\{(1-t)\R_i W_0(\xi-t a\otimes b)+t\R_i W_0(\xi+(1-t)a\otimes b)\big\}.
$$
 \begin{lemma}[\cite{kohn-strang86}]\label{LeMmA-A3}
$\R_{i+1} W_0\leq \R_iW_0\hbox{ for all }i\geq 0\hbox{ and }\R W_0=\inf_{i\geq 0} \R_i W_0$.
\end{lemma}
Fix any $i\geq 0$ and any $\psi\in\Aff_{\rm li}^{\rm reg}(\Sigma;\RR^3)$.  Then, there exists a finite family $\{V_j\}_{j\in J}$ of open disjoint subsets of $\Sigma$ such that $|\Sigma\setminus\cup_{j\in J}V_j|=0$ for all $j\in J$ and for every  $j\in J$, $|\partial V_j|=0$ and $\nabla \psi(x)=\xi_j$ in $V_j$ with $\xi_j\in\MM^{3\times 2}$. As $\psi$ is locally injective we have ${\rm rang}(\xi_j)=2$ for all $j\in J$. Fix any $j\in J$. 
\begin{lemma}\label{LeMmA-A4}
$\R_i W_0$ is continuous.
\end{lemma}
\begin{lemma}\label{LeMmA-A5}
There exist $a\in\RR^2$, $b\in\RR^3$ and $t\in[0,1]$ such that
$$
\R_{i+1}W_0(\xi_j)=(1-t)\R_i W_0(\xi_j-t a\otimes b)+t\R_i W_0(\xi_j+(1-t)a\otimes b) $$
with $a\otimes b\in\RR^2\otimes\RR^3\subset\MM^{3\times 2}$ given by $(a\otimes b)x:=\langle a,x\rangle b$ for all $x\in\RR^2$, where $\langle\cdot,\cdot\rangle$ denotes the scalar product in  $\RR^2$. 
\end{lemma}
Without loss of generality we can assume that $a=(1,0)$. For every $n\geq 1$ and for every $k\in\{0,\cdots,n-1\}$, consider $A^-_{k,n}$, $A^+_{k,n}$, $B_{k,n}$, $B^-_{k,n}$, $B^+_{k,n}$, $C_{k,n}$, $C^-_{k,n}$, $C^+_{k,n}\subset Y$ given by:
\begin{itemize}
\item[]$A^-_{k,n}:=\big\{(x_1,x_2)\in Y:{k\over n}\leq x_1\leq{k\over n}+{1-t\over n}\hbox{ and }{1\over n}\leq x_2\leq 1-{1\over n}\big\}$;
\item[]$A^+_{k,n}:=\big\{(x_1,x_2)\in Y:{k\over n}+{1-t\over n}\leq x_1\leq {k+1\over n}\hbox{ and }{1\over n}\leq x_2\leq 1-{1\over n}\big\}$;
\item[]$B_{k,n}:=\big\{(x_1,x_2)\in Y:{k\over n}\leq x_1\leq{k+1\over n}\hbox{ and }0\leq x_2\leq -x_1+{k+1\over n}\big\}$;
\item[]$B^-_{k,n}:=\big\{(x_1,x_2)\in Y:-x_2+{k+1\over n}\leq x_1\leq-tx_2+{k+1\over n}\hbox{ and }0\leq x_2\leq {1\over n}\big\}$;
\item[]$B^+_{k,n}:=\big\{(x_1,x_2)\in Y:-tx_2+{k+1\over n}\leq x_1\leq {k+1\over n}\hbox{ and }0\leq x_2\leq {1\over n}\big\}$;
\item[]$C_{k,n}:=\big\{(x_1,x_2)\in Y:{k\over n}\leq x_1\leq{k+1\over n}\hbox{ and }x_1+1-{k+1\over n}\leq x_2\leq 1\big\}$;
\item[]$C^-_{k,n}:=\big\{(x_1,x_2)\in Y:x_2-1+{k+1\over n}\leq x_1\leq t(x_2-1)+{k+1\over n}\hbox{ and }{n-1\over n}\leq x_2\leq 1\big\}$; 
\item[]$C^+_{k,n}:=\big\{(x_1,x_2)\in Y:t(x_2-1)+{k+1\over n}\leq x_1\leq {k+1\over n}\hbox{ and }{n-1\over n}\leq x_2\leq 1\big\}$,
\end{itemize}
and define  $\{\sigma_{n}\}_{n\geq 1}\subset\Aff_0^{\rm reg}(Y;\RR)$ by
$$
\sigma_{n}(x_1,x_2):=\left\{
\begin{array}{ll}
-t(x_1-{k\over n})&\hbox{if }(x_1,x_2)\in A^-_{k,n}\\
(1-t)(x_1-{k+1\over n})&\hbox{if }(x_1,x_2)\in A^+_{k,n}\cup B^+_{k,n}\cup C^+_{k,n}\\
-t(x_1+x_2-{k+1\over n})&\hbox{if }(x_1,x_2)\in B^-_{k,n}\\
-t(x_1-x_2+1-{k+1\over n})&\hbox{if }(x_1,x_2)\in C^-_{k,n}\\
0&\hbox{if }(x_1,x_2)\in B_{k,n}\cup C_{k,n}.
\end{array}
\right.
$$
Set
$$
b_{\ell}:=\left\{
\begin{array}{ll}
b&\hbox{if }b\not\in{\rm Im}\xi_j\\
b+{1\over\ell}\nu&\hbox{if }b\in{\rm Im}\xi_j
\end{array}
\right.
$$
(with ${\rm Im}\xi_j:=\{\xi_j\cdot x:x\in\RR^2\}\subset\RR^3$) where $\ell\geq 1$ and $\nu\in\RR^3$ is a normal vector to  ${\rm Im}\xi_j$. 
\begin{lemma}\label{LeMmA-A6}
Define $\{\theta_{n,\ell}\}_{n,\ell\geq 1}\subset\Aff_0^{\rm reg}(Y;\RR^3)$ by
$$
\theta_{n,\ell}(x):=\sigma_{n}(x)b_{\ell}.
$$
Then
\begin{equation}\label{EquBBLeMMa2}
\lim_{\ell\to+\infty}\lim_{n\to+\infty}\int_Y\R_iW_0(\xi_j+\nabla \theta_{n,\ell}(x))dx=\R_{i+1}W_0(\xi_j).
\end{equation}
\end{lemma}
Consider $V^j_{q}\subset V_j$ given by $V^j_{q}:=\{x\in V_j:{\rm dist}(x,\partial V_j)>{1\over q}\}$ with $q\geq 1$ large enough. Then, there exists a finite family   $\{r_{m}+\rho_{m}Y\}_{m\in M}$ of disjoint subsets of $V^j_{q}$ with $r_{m}\in\RR^2$ and $\rho_{m}\in]0,1[$, such that
$
|V^j_{q}\setminus\cup_{m\in M}(r_{m}+\rho_{m}Y)|\leq{1\over q}.
$ 

Let $\{\phi_{n,\ell,q}\}_{n,\ell,q\geq 1}\subset\Aff_0^{\rm reg}(V_j;\RR^3)$ be given by
$$
\phi_{n,\ell,q}(x):=\left\{
\begin{array}{ll}
\displaystyle\rho_{m}\theta_{n,\ell}\left({x-r_{m}\over \rho_{m}}\right)&\hbox{si }x\in r_{m}+\rho_{m}Y\subset V^j_{q}\\
0&\hbox{si }x\in V_j\setminus V^j_{q}.
\end{array}
\right.
$$
\begin{lemma}\label{LeMmA-A7}
Define $\{\Phi^j_{n,\ell,q}\}_{n,l,q}\subset\Aff^{\rm reg}(V_j;\RR^3)$ by
\begin{equation}\label{BBFunct2}
\Phi^j_{n,\ell,q}(x):=\psi(x)+\phi_{n,\ell,q}(x).
\end{equation}
Then{\rm :}
\begin{itemize}
\item[(i)] $\hbox{for every }n,\ell,q\geq 1,\  \Phi^j_{n,\ell,q}\hbox{ is locally injective;}$
\item[(ii)] $\hbox{for every } \ell,q\geq 1,\ \Phi^j_{n,\ell,q}\to \psi\hbox{ in }L^{p}(V_j;\RR^3);$
\item[(iii)] $\lim_{q\to+\infty}\lim_{\ell\to+\infty}\lim_{n\to+\infty}\int_{V_j}\R_iW_0(\nabla \Phi^j_{n,\ell,q}(x))dx=|V_j|\R_{i+1}W_0(\xi_j)$.
\end{itemize}
\end{lemma}
We can now prove (\ref{D1bis}). According to Lemma \ref{LeMmA-A3}, it is sufficient to show that 
$$
\overline{\mathcal{I}}_{\Aff_{\rm li}^{\rm reg}}(\psi)\leq \int_\Sigma\R_iW_0(\nabla \psi(x))dx\hbox{ for all }\psi\in\Aff_{\rm li}^{\rm reg}(\Sigma;\RR^3)\leqno (P_i)
$$
for all $i\geq 0$. The proof is by induction on $i$. As $R_0 W_0=W_0$ it is clear that $(P_0)$ is true.  Assume that $(P_i)$ is true, and prove that $(P_{i+1})$ is true. Let $\psi\in\Aff_{\rm li}^{\rm reg}(\Sigma;\RR^3)$. Then, there exists a finite family $\{V_j\}_{j\in J}$ of open disjoint subsets of $\Sigma$ such that $|\Sigma\setminus\cup_{j\in J}V_j|=0$ for all $j\in J$ and for every  $j\in J$, $|\partial V_j|=0$ and $\nabla \psi(x)=\xi_j$ in $V_j$ with $\xi_j\in\MM^{3\times 2}$.  Define $\{\Psi_{n,\ell,q}\}_{n,\ell,q\geq 1}\subset\Aff^{\rm reg}(\Sigma;\RR^3)$ by
$$
\Psi_{n,\ell,q}(x):=\Phi^j_{n,\ell,q}(x)\hbox{ if }x\in V_j
$$
with $\Phi^j_{n,\ell,q}$ given by (\ref{BBFunct2}). Taking Lemma \ref{LeMmA-A7}(i) into account (and recalling that rappelant $\psi$ is locally injective) it is easy to see that $\Psi_{n,\ell,q}$ is locally injective.  Using $(P_i)$ we can assert that 
$$
\overline{\mathcal{I}}_{\Aff_{\rm li}^{\rm reg}}(\Psi_{n,\ell,q})\leq\int_{\Sigma}\R_iW_0(\nabla\Psi_{n,\ell,q}(x))dx
$$
for all $n,\ell,q\geq 1$. By Lemma \ref{LeMmA-A7}(ii) it is clear that for every $\ell,q\geq 1$, $\Psi_{n,l,q}\to \psi$ in $L^{p}(\Sigma;\RR^3)$. It follows that 
$$
\overline{\mathcal{I}}_{\Aff_{\rm li}^{\rm reg}}(\psi)\leq \lim_{n\to+\infty}\overline{\mathcal{I}}_{\Aff_{\rm li}^{\rm reg}}(\Psi_{n,\ell,q})\leq\lim_{n\to+\infty}\int_{\Sigma}\R_iW_0(\nabla\Psi_{n,\ell,q}(x))dx
$$
for all $\ell,q\geq 1$. Moreover, from Lemma \ref{LeMmA-A7}(iii) we see that
$$
\lim_{q\to+\infty}\lim_{\ell\to+\infty}\lim_{n\to+\infty}\int_{\Sigma}\R_iW_0(\nabla\Psi_{n,\ell,q}(x))dx=\int_\Sigma\R_{i+1}W_0(\nabla \psi(x))dx.
$$
Hence 
$$
\overline{\mathcal{I}}_{\Aff_{\rm li}^{\rm reg}}(\psi)\leq \int_\Sigma\R_{i+1}W_0(\nabla \psi(x))dx,
$$
and ($P_{i+1}$) follows. This completes the proof of the assertion (\ref{D1bis}).
\end{proof}

\medskip

In what follows, we give the proof of Lemmas \ref{LeMmA-A6}, \ref{LeMmA-A7}, \ref{LeMmA-A4} and \ref{LeMmA-A5}.


\begin{proof}[\bf Proof of Lemma \ref{LeMmA-A6}]
Recalling that $a=(1,0)$ we see that
$$
\xi_j+\nabla\theta_{n,\ell}(x):=\left\{
\begin{array}{ll}
\xi_j-ta\otimes b_\ell&\hbox{if }x\in {\rm int}(A^-_{k,n})\\
\xi_j+(1-t)a\otimes b_\ell&\hbox{if }x\in {\rm int}(A^+_{k,n}\cup B^+_{k,n}\cup C^+_{k,n})\\
\xi_j-t(a+a^\perp)\otimes b_\ell&\hbox{if }x\in {\rm int}(B^-_{k,n})\\
\xi_j-t(a-a^\perp)\otimes b_\ell&\hbox{if }x\in {\rm int}(C^-_{k,n})\\
\xi_j&\hbox{if }x\in {\rm int}(B_{k,n})\cup{\rm int}(C_{k,n})
\end{array}
\right.
$$
with $a^\perp=(0,1)$ (and ${\rm int}(E)$ denotes the interior of the set $E$). Moreover, we have:
\begin{itemize}
\item[]${\displaystyle\int_{\cup_{k=0}^{n-1} A^-_{k,n}}\R_iW_0(\xi_j-ta\otimes b_\ell)dx}=(1-t)(1-{2\over n})\R_iW_0(\xi_j-ta\otimes b_\ell)$;
\item[]${\displaystyle\int_{\cup_{k=0}^{n-1} A^+_{k,n}}\R_i W_0(\xi_j+(1-t)a\otimes b_\ell)dx}=t(1-{2\over n})\R_i W_0(\xi_j+(1-t)a\otimes b_\ell)$;
\item[]${\displaystyle\int_{\cup_{k=0}^{n-1} (B^+_{k,n}\cup C^+_{k,n})}\R_i W_0(\xi_j+(1-t)a\otimes b_\ell)dx}={t\over n}\R_i W_0(\xi_j+(1-t)a\otimes b_\ell)$;
\item[]${\displaystyle\int_{\cup_{k=0}^{n-1} B^-_{k,n}}\R_i W_0(\xi_j-t(a+a^\perp)\otimes b_\ell)dx}={1-t\over 2n}\R_i W_0(\xi_j-t(a+a^\perp)\otimes b_\ell)$;
\item[]${\displaystyle\int_{\cup_{k=0}^{n-1} C^-_{k,n}}\R_i W_0(\xi_j-t(a-a^\perp)\otimes b_\ell)dx}={1-t\over 2n}\R_i W_0(\xi_j-t(a-a^\perp)\otimes b_\ell)$;
\item[]${\displaystyle\int_{\cup_{k=0}^{n-1} (B_{k,n}\cup C_{k,n})}\R_i W_0(\xi_j)dx}={1\over n}\R_i W_0(\xi_j)$.
\end{itemize}
Hence 
\begin{eqnarray*}
\int_Y\R_i W_0(\xi_j+\nabla\theta_{n,\ell}(x))dx&=&\left(1-{2\over n}\right)\Big[(1-t)\R_iW_0(\xi_j-ta\otimes b_\ell)+t\R_i W_0(\xi_j\\
&&+\ (1-t)a\otimes b_\ell)\Big]+{1\over n}\Big[t\R_i W_0(\xi_j+(1-t)a\otimes b_\ell)\\
&&+\ {1-t\over 2}\big(\R_i W_0(\xi_j-t(a+a^\perp)\otimes b_\ell)+\R_i W_0(\xi_j-\\
&&t(a-a^\perp)\otimes b_\ell)\big)+\R_i W_0(\xi_j)\Big]
\end{eqnarray*}
for all $n,\ell\geq 1$. It follows that
\begin{eqnarray*}
\lim_{n\to+\infty}\int_Y\R_i W_0(\xi_j+\nabla\theta_{n,\ell}(x))dx&=&(1-t)\R_iW_0(\xi_j-ta\otimes b_\ell)\\
&&+\ t\R_i W_0(\xi_j+(1-t)a\otimes b_\ell)
\end{eqnarray*}
for all $\ell\geq 1$. Taking Lemma \ref{LeMmA-A4} into account and noticing that  $b_\ell\to b$, we deduce that 
\begin{eqnarray*}
\lim_{\ell\to+\infty}\lim_{n\to+\infty}\int_Y\R_i W_0(\xi_j+\nabla\theta_{n,\ell}(x))dx&=&(1-t)\R_iW_0(\xi_j-ta\otimes b)\\
&&+\ t\R_i W_0(\xi_j+(1-t)a\otimes b),
\end{eqnarray*}
and (\ref{EquBBLeMMa2}) follows by using Lemma \ref{LeMmA-A5}.
\end{proof}

\begin{proof}[\bf Proof of Lemma \ref{LeMmA-A7}]
(i) Let $x\in V_j$ and let $W\subset V_j$ be the connected component of $V_j$ such that $x\in W$ (As $V_j$ is open, so is $W$). Since $\nabla \psi=\xi_j$ in $W$, there exists $c\in\RR^3$ such that $\psi(x^\prime)=\xi_j\cdot x^\prime+c$ for all $x^\prime\in W$.   We claim that ${\Phi^j_{n,\ell,q}}{\lfloor_{W}}$ is injective. Indeed, let $x^\prime\in W$ be such that $\Phi^j_{n,\ell,q}(x)=\Phi^j_{n,\ell,q}(x^\prime)$. Then, one of the three possibilities holds: 
\begin{eqnarray}
&&\left\{\begin{array}{l}\Phi^j_{n,\ell,q}(x)=\xi_j\cdot x+c+\rho_m\sigma_{n}\left({x-r_m\over\rho_m}\right)b_\ell\\
\Phi^j_{n,\ell,q}(x^\prime)=\xi_j\cdot x^\prime+c+\rho_{m^\prime}\sigma_{n}\left({x^\prime-r_{m^\prime}\over\rho_{m^\prime}}\right)b_\ell;\end{array}\right.\label{EQUaT(4)BBLemmA1}\\
&&\left\{\begin{array}{l}\Phi^j_{n,\ell,q}(x)=\xi_j\cdot x+c+\rho_m\sigma_{n}\left({x-r_m\over\rho_m}\right)b_\ell\\
\Phi^j_{n,\ell,q}(x^\prime)=\xi_j\cdot x^\prime+c;
\end{array}\right.\label{EQUaT(5)BBLemmA1}\\
&&\left\{\begin{array}{l}\Phi^j_{n,\ell,q}(x)=\xi_j\cdot x+c\\
\Phi^j_{n,\ell,q}(x^\prime)=\xi_j\cdot x^\prime+c.
\end{array}\right.\label{EQUaT(6)BBLemmA1}
\end{eqnarray}
Setting $\alpha:=\rho_m\sigma_n({x-r_m\over\rho_m})-\rho_{m^\prime}\sigma_n({x^\prime-r_{m^\prime}\over \rho_{m^\prime}})$ and $\beta:=\rho_m\sigma_n({x-r_m\over\rho_m})$ we have: 
\begin{itemize}
\item[\SMALL$\blacklozenge$]$\left\{\begin{array}{ll}
\xi_j(x^\prime-x)=0&\hbox{if }\alpha=0\\
b_\ell={1\over \alpha}\xi_j(x^\prime-x)&\hbox{if }\alpha\not=0
\end{array}
\right.
\hbox{ when (\ref{EQUaT(4)BBLemmA1}) is satisfied;}$
\item[\SMALL$\blacklozenge$]$\left\{
\begin{array}{ll}
\xi_j(x^\prime-x)=0&\hbox{if }\beta=0\\
b_\ell={1\over \beta}\xi_j(x^\prime-x)&\hbox{if }\beta\not=0
\end{array}
\right.
\hbox{ when (\ref{EQUaT(5)BBLemmA1}) is satisfied;}$
\item[\SMALL$\blacklozenge$]\ $\xi_j(x^\prime-x)=0\hbox{ when (\ref{EQUaT(6)BBLemmA1}) is satisfied.}$
\end{itemize}
It follows that if $x\not=x^\prime$ then either ${\rm rank}(\xi_j)<2$ or $b_\ell\in{\rm Im}\xi_j$ which is impossible. Hence $x=x^\prime$. 

(ii) Given $\ell,q\geq 1$, we have $\|\phi_{n,\ell,q}\|_{L^\infty(V_j;\RR^3)}\leq \|\theta_{n,\ell}\|_{L^\infty(Y;\RR^3)}=|b_l|\|\sigma_n\|_{L^\infty(Y;\RR)}$.  On the other hand, for every $k\in\{0,\cdots,n-1\}$, it is clear that $|\sigma_n(x)|\leq {t(1-t)\over n}$ for all $x\in]{k\over n},{k+1\over n}[\times]0,1[$, and so $\sigma_n\to 0$ in $L^\infty(Y;\RR)$. Hence $\phi_{n,\ell,q}\to 0$ in $L^\infty(V_j;\RR^3)$, and (ii) follows.

(iii) Recalling that $\phi_{n,\ell,q}=0$ in $V_j\setminus\hat{V}^j_q$ and $\sum_{m\in M}\rho_m^2=|\hat{V}^j_q|$ we see that 
\begin{eqnarray*}
\int_{V_j}\R_iW_0(\nabla \Phi^j_{n,\ell,q}(x))dx\hskip-2mm&=&\hskip-2mm\int_{V_j}\R_iW_0(\xi_j+\nabla \phi_{n,\ell,q}(x))dx\\
&=&\hskip-2mm\int_{\hat{V}^j_q}\R_iW_0(\xi_j+\nabla \phi_{n,\ell,q}(x))dx+|V_j\setminus \hat{V}^j_q|\R_iW_0(\xi_j)\\
&=&\hskip-2mm|\hat{V}^j_q|\int_Y\R_iW_0(\xi_j+\nabla\theta_{n,\ell}(x))dx+|V_j\setminus \hat{V}^j_q|\R_iW_0(\xi_j).
\end{eqnarray*}
Using Lemma \ref{LeMmA-A6} we deduce that 
$$
\lim_{\ell\to+\infty}\lim_{n\to+\infty}\int_{V_j}\R_iW_0(\nabla \Phi^j_{n,\ell,q}(x))dx=|\hat{V}^j_q|\R_{i+1}W_0(\xi_j)+|V_j\setminus \hat{V}^j_q|\R_iW_0(\xi_j)
$$
for all $q\geq 1$, and (iii) follows since  $|\hat{V}^j_q|=|V^j_q|-|V^j_q\setminus \hat{V}^j_q|\to|V_j|$ (because $|V^j_q|\to|V_j|$ and ${1\over q}\geq|V^j_q\setminus \hat{V}^j_q|\to 0$) and $|V_j\setminus \hat{V}^j_q|=|V_j\setminus V^j_q|+|V^j_q\setminus\hat{V}^j_q|\to 0$ (because $|V_j\setminus V^j_q|\to 0$). 
\end{proof}

\begin{proof}[\bf Proof of Lemmas \ref{LeMmA-A4} and \ref{LeMmA-A5}]
We begin by proving three lemmas.
\begin{lemma}\label{ProduitTensorielEstFermDanslesMatrices}
$\RR^2\otimes\RR^3$ is closed in $\MM^{3\times 2}$. 
\end{lemma}
\begin{proof}[\bf Proof of Lemma \ref{ProduitTensorielEstFermDanslesMatrices}]
Let $\{a_n\otimes b_n\}_{n\geq 1}\subset\RR^2\otimes\RR^3$ and let  $\xi\in\MM^{3\times 2}$ be such that $a_n\otimes b_n\to \xi$. For every $n\geq 1$, $a_n\otimes b_n=u_n\otimes v_n$ with $u_n={a_n\over|a_n|}\in\SS^1$ and $v_n=|a_n| b_n$, where $\SS^1$ is the unit sphere in $\RR^2$. As $\SS^1$ is compact, there exists $u\in\SS^1$ such that (up to a subsequence) $u_n\to u$.  Let $u_0\in\RR^2$ be such that $\langle u,u_0\rangle\not=0$. Then,  $\langle u_n, u_0\rangle\not=0$ for all $n\geq n_0$ with $n_0\geq 1$ large enough. For every $n\geq n_0$, $v_n={1\over \langle u_n,u_0\rangle}(u_n\otimes v_n)u_0$, and so $v_n\to{1\over \langle u,u_0\rangle}\xi u_0=:v\in\RR^3$. It follows that $a_n\otimes b_n\to u\otimes v$. Hence $\xi=u\otimes v$.
\end{proof}

\medskip

Let $\mathcal{H}:\MM^{3\times 2}\to[0,+\infty]$ be defined by 
$$
\mathcal{H}(\xi):=\inf\Big\{H(\xi,t,a\otimes b):(t,a\otimes b)\in[0,1]\times\RR^2\otimes\RR^3\Big\},
$$
where $H:\MM^{3\times 2}\times[0,1]\times\RR^2\otimes\RR^3\to[0,+\infty]$ is given by 
$$
H(\xi,t,a\otimes b):=(1-t)h(\xi-ta\otimes b)+th(\xi+(1-t)a\otimes b)
$$
with $h:\MM^{3\times 2}\to[0,+\infty]$ continuous and coercive. 
\begin{lemma}\label{KhonStrangFormulaProofContinuitetInfAtteint}
Given $\xi\in\MM^{3\times 2}$, if $\mathcal{H}(\xi)<+\infty$ then there exists $(t,a\otimes b)\in[0,1]\times\RR^2\otimes\RR^3$ such that $\mathcal{H}(\xi)=H(\xi,t,a\otimes b)$.
\end{lemma}
\begin{proof}[\bf Proof of Lemma \ref{KhonStrangFormulaProofContinuitetInfAtteint}]
Let $\{(t_n,a_n\otimes b_n)\}_{n\geq 1}\subset[0,1]\times\RR^2\otimes\RR^3$ be a minimizing sequence for $\mathcal{H}(\xi)$ such that $t_n\to t\in[0,1]$. Set $F_n:=\xi-t_na_n\otimes b_n$ et $G_n:=\xi+(1-t_n)a_n\otimes b_n$. Then $(1-t_n)F_n+t_nG_n=\xi$ et $G_n-F_n=a_n\otimes b_n$ for all $n\geq 1$. By the coercivity of  $h$ we have
\begin{eqnarray}\label{KhonStrangFormulaProofContinuitetInfAtteintCond1}
(1-t_n)|F_n|^p+t_n|G_n|^p\leq c\hbox{ for all }n\geq 1\hbox{ and some } c>0.
\end{eqnarray}
One of the two possibilities holds:
\begin{itemize}
\item[\SMALL$\blacklozenge$] $t\in]0,1[$;
\item[\SMALL$\blacklozenge$] either $t=0$ or $t=1$.
\end{itemize}
{\em Case where $t\in]0,1[$.} It is clear that $1-t_n\geq\alpha_1>0$ et $t_n\geq\alpha_2>0$ for all $n\geq 1$. Using (\ref{KhonStrangFormulaProofContinuitetInfAtteintCond1}) we deduce that there exists   $F,G\in\MM^{3\times 2}$ such that (up to a subsequence) $F_n\to F$ and $G_n\to G$.  Consequently,  $G_n-F_n=a_n\otimes b_n\to G-F$. But, from Lemma \ref{ProduitTensorielEstFermDanslesMatrices}, $\RR^2\otimes\RR^3$ is closed in $\MM^{3\times 2}$, and so $G-F\in\RR^2\otimes\RR^3$, i.e., $G-F=a\otimes b$ with $a\in\RR^2$ et $b\in\RR^3$. As $H(\xi,\cdot,\cdot)$ is continuous, it follows that 
$$
\mathcal{H}(\xi)=\lim_{n\to+\infty}H(\xi,t_n,a_n\otimes b_n)=H(\xi,t,a\otimes b).
$$

\smallskip

{\em Case where either $t=0$ or $t=1$.} Assume that $t=0$ (the case $t=1$ can be treated in the same way). Then $1-t_n\geq\alpha>0$ for all $n\geq 1$. As $p>1$ and  $t_n\to 0$, using (\ref{KhonStrangFormulaProofContinuitetInfAtteintCond1}) we deduce that there exists $F\in\MM^{3\times 2}$ such that $F_n\to F$ and $t_n G_n\to 0$. As $(1-t_n)F_n+t_nG_n=\xi$ for all $n\geq 1$, it follows that $F=\xi$. Hence
$$
\lim_{n\to+\infty}(1-t_n)h(F_n)=h(\xi)
$$ 
since  $h$ is continuous. But $t_n h(G_n)=H(\xi,t_n,a_n\otimes b_n)-(1-t_n)h(F_n)$ for all $n\geq 1$ and $\mathcal{H}(\xi)\leq h(\xi)$, hence 
$$
\lim_{n\to+\infty}t_n h(G_n)=\mathcal{H}(\xi)-h(\xi)\leq 0.
$$ 
On the other hand, using the coercivity of $h$, we see that  $t_nh(G_n)\geq Ct_n|G_n|^p$ for all $n\geq 1$ and some $C>0$. Then 
$$
\lim_{n\to+\infty}t_nh(G_n)\geq C\lim_{n\to+\infty}t_n|G_n|^p=0,
$$ 
and consequently
$$
\lim_{n\to+\infty}t_nh(G_n)=0.
$$ 
Thus $\mathcal{H}(\xi)=h(\xi)=H(\xi,0,a\otimes b)$, where  $a\otimes b$ is any element of $\RR^2\otimes\RR^3$. 
\end{proof}

\begin{lemma}\label{KhonStrangFormulaProofContinuitetInfAtteint2}
$\mathcal{H}$ is continuous and coercive. 
\end{lemma}
\begin{proof}[\bf Proof of Lemma \ref{KhonStrangFormulaProofContinuitetInfAtteint2}]
We first prove that $\mathcal{H}$ is continuous. Since $H(\cdot,t,a\otimes b)$ is continuous for all $(t,a\otimes b)\in[0,1]\times \RR^2\otimes\RR^3$, $\mathcal{H}$ is upper semicontinuous. Thus, we are reduced to show that  $\mathcal{H}$ is lower semicontinuous. To do this, consider $\xi\in\MM^{3\times 2}$ and $\{\xi_n\}_{n\geq 1}\subset\MM^{3\times 2}$ such that:
\begin{itemize}
\item[\SMALL$\blacklozenge$] $\xi_n\to\xi$;  
\item[\SMALL$\blacklozenge$] $\sup_{n\geq 1}\mathcal{H}(\xi_n)<+\infty$;
\item[\SMALL$\blacklozenge$] $\lim_{n\to+\infty}\mathcal{H}(\xi_n)=\liminf_{n\to+\infty}\mathcal{H}(\xi_n)$, 
\end{itemize}
and prove that  
$$
\mathcal{H}(\xi)\leq\lim_{n\to+\infty}\mathcal{H}(\xi_n).
$$ 
By Lemma \ref{KhonStrangFormulaProofContinuitetInfAtteint}, for every $n\geq 1$, there exists $(t_n,a_n\otimes b_n)\in[0,1]\times\RR^2\otimes\RR^3$ such that $\mathcal{H}(\xi_n)=H(\xi_n,t_n,a_n\otimes b_n)$. Without loss of generality we can assume that $t_n\to t\in[0,1]$. From the coercivity of  $h$, we deduce that (\ref{KhonStrangFormulaProofContinuitetInfAtteintCond1}) holds with  $F_n:=\xi_n-t_na_n\otimes b_n$ and $G_n:=\xi_n+(1-t_n)a_n\otimes b_n$. As in the proof of Lemma \ref{KhonStrangFormulaProofContinuitetInfAtteint}, we consider two cases.

\smallskip

{\em Case where $t\in]0,1[$.} Using the same arguments as in the proof of Lemma \ref{KhonStrangFormulaProofContinuitetInfAtteint}, we obtain $G_n-F_n=a_n\otimes b_n\to a\otimes b$ with $a\in\RR^2$ and $b\in\RR^3$. Hence
$$
\lim_{n\to+\infty}\mathcal{H}(\xi_n)=\lim_{n\to+\infty}H(\xi_n,t_n,a_n\otimes b_n)=H(\xi,t,a\otimes b)\geq \mathcal{H}(\xi)
$$ 
since $H$ is continuous.

\smallskip

{\em Case where either $t=0$ or $t=1$.} Assume that $t=1$ (the case $t=0$ can be treated in the same way). Then $t_n\geq\beta>0$ for all $n\geq 1$. As $p>1$ and $t_n\to 1$, by (\ref{KhonStrangFormulaProofContinuitetInfAtteintCond1}) we have $G_n\to G$ with $G\in\MM^{3\times 2}$ and $(1-t_n)F_n\to 0$. As $(1-t_n)F_n+t_nG_n=\xi_n$ for all $n\geq 1$. Hence 
$$
G=\lim_{n\to+\infty}(1-t_n)F_n+t_nG_n=\lim_{n\to+\infty}\xi_n=\xi,
$$ 
and consequently 
$$
\lim_{n\to+\infty}t_nh(G_n)=h(\xi)
$$ 
since $h$ is continuous. But $(1-t_n)h(F_n)=H(\xi_n,t_n,a_n\otimes b_n)-t_nh(G_n)$ for all $n\geq 1$, hence 
$$
\lim_{n\to+\infty}(1-t_n)h(F_n)=\lim_{n\to+\infty}\mathcal{H}(\xi_n)-h(\xi)\leq 0
$$
because $\lim_{n\to+\infty}\mathcal{H}(\xi_n)\leq h(\xi)$ since $\mathcal{H}(\xi_n)\leq h(\xi_n)$ pour tout $n\geq 1$. On the other hand, using the coercivity of  $h$, we see that $(1-t_n)h(F_n)\geq C(1-t_n)|F_n|^p$ for all $n\geq 1$ with $C>0$. Hence 
$$
\lim_{n\to+\infty}(1-t_n)h(F_n)\geq C\lim_{n\to+\infty}(1-t_n)|F_n|^p=0.
$$
Thus $\lim_{n\to+\infty}(1-t_n)h(F_n)=0$, and consequently 
$$
\lim_{n\to+\infty}\mathcal{H}(\xi_n)=h(\xi)\geq \mathcal{H}(\xi). 
$$

\smallskip

We prove now that $\mathcal{H}$ is coercive. By the coercivity of $h$ we have  $$
\mathcal{H}(\xi)\geq C\inf\{(1-t)|\xi-ta\otimes b|^p+t|\xi+(1-t)a\otimes b|^p:(t,a\otimes b)\in[0,1]\times\RR^2\otimes\RR^3\}
$$ 
for all $\xi\in\MM^{3\times 2}$ and some $C>0$. But 
$$
(1-t)|\xi-ta\otimes b|^p+t|\xi+(1-t)a\otimes b|^p\geq|(1-t)(\xi-ta\otimes b)+t(\xi+(1-t)a\otimes b)|^p=|\xi|^p,
$$ 
and so $\mathcal{H}(\xi)\geq C|\xi|^p$ for all $\xi\in\MM^{3\times 2}$.
\end{proof}

\medskip

We can now prove Lemmas \ref{LeMmA-A4} and \ref{LeMmA-A5}. As $W$ is continuous and coercive, it is easy to see that $W_0=\mathcal{R}_0W_0$ is also continuous and coercive. Moreover, using Lemma \ref{KhonStrangFormulaProofContinuitetInfAtteint2} with $h=\R_{q}W_0$, we see that if $\R_q W_0$ is continuous and coercive, so is $\R_{q+1}W_0$. Hence, $\R_q W_0$ is continuous and coercive for all $q\geq 0$, which proves Lemma \ref{LeMmA-A4}. As $\rank(\xi_j)=2$, by (\ref{P01}) we have $W_0(\xi_j)<+\infty$. Hence,  $\R_{i+1}W_0(\xi_j)<+\infty$ since $\R_{i+1}W_0\leq W_0$, and Lemma \ref{LeMmA-A5} follows by using Lemma \ref{KhonStrangFormulaProofContinuitetInfAtteint} with  $h=\R_iW_0$.
\end{proof}

\end{document}